\documentclass[a4paper,10pt]{article}

\usepackage[english]{babel}
\usepackage{amsmath}
\usepackage{amsthm}
\usepackage{amssymb}
\usepackage{amsfonts}
\usepackage{amsxtra}
\usepackage{color}
\usepackage[breaklinks]{hyperref}
\usepackage{enumitem}
\usepackage{tocloft}
\usepackage{makeidx}

\setenumerate{topsep=0pt, itemsep=0pt}
\setitemize{topsep=0pt, itemsep=0pt}
\setdescription{topsep=0pt, itemsep=0pt}
\setlist{noitemsep}
\renewcommand{\phi}{\varphi}
\renewcommand{\epsilon}{\varepsilon}

\theoremstyle{plain}

\theoremstyle{plain}
    \newtheorem{thm}{Theorem}[section]
    \newtheorem{lem}[thm]{Lemma}
    
    \newtheorem{prop}[thm]{Proposition}
    
\theoremstyle{definition}
    \newtheorem{defn}[thm]{Definition}
    
    \newtheorem{rem}[thm]{Remark}

\theoremstyle{remark}

\newcommand{\Ad}{\textup{Ad}}

\newcommand{\id}{\textup{id}}

\newcommand {\rr}{{\mathbb{R}}}
\newcommand {\cc}{{\mathbb{C}}}

\newcommand {\nn}{{\mathbb{N}}}

\newcommand {\g}{{\mathfrak{g}}}

\newcommand {\la}{{\mathfrak{a}}}

\newcommand {\lp}{{\mathfrak{p}}}

\newcommand {\lk}{{\mathfrak{k}}}

\newcommand {\lnn}{{\mathfrak{n}}}

\renewcommand{\phi}{\varphi}

\DeclareMathOperator{\Ind}{Ind}

\DeclareMathOperator{\Op}{Op}

\DeclareMathOperator{\sign}{sign}

\def\dbar{{\ \mathchar'26\mkern-12mu d}}

\numberwithin{equation}{section}

\title{Patterson--Sullivan distributions for rank one symmetric spaces of the noncompact type}
\author{Joachim Hilgert, Michael Schr\"oder\footnote{Supported by the
International Research Training Group DFG-1133 ``Geometry and
Analysis of Symmetries''}}

\makeindex

\begin{document}

\maketitle

\begin{abstract}

There is a remarkable relation between two kinds of phase space
distributions associated to eigenfunctions of the Laplacian of a
compact hyperbolic manifold: It was observed in \cite{AZ} that for
compact hyperbolic surfaces $X_{\Gamma}=\Gamma\backslash\mathbb{H}$
Wigner distributions $\int_{S^* X_{\Gamma}} a \, dW_{ir_j} = \langle
\mathrm{Op}(a)\phi_{ir_j},\phi_{ir_j}\rangle_{L^2(X_{\Gamma})}$ and
Patterson--Sullivan distributions $PS_{ir_j}$ are asymptotically
equivalent as $r_j\rightarrow\infty$. We generalize the definitions
of these distributions to all rank one symmetric spaces of
noncompact type and introduce off-diagonal elements
$PS_{\lambda_j,\lambda_k}$. Further, we give explicit relations
between off-diagonal Patterson--Sullivan distributions and
off-diagonal Wigner distributions and describe the asymptotic
relation between
these distributions.\\

\textit{2000 Mathematics Subject Classification:} Primary 53C35; Secondary  58C40, 58J50.\\

\textit{Key words and phrases:} Patterson--Sullivan distributions,
Wigner distributions, quantum ergodicity, rank one symmetric spaces,
non-Euclidean pseudo-differential analysis.

\end{abstract}

\tableofcontents

\section{Introduction}
In this paper we generalize an interesting link between two kinds of
phase space distributions which was observed in \cite{AZ} for
hyperbolic surfaces to all rank one Riemannian symmetric spaces of
the noncompact type. The distributions of interest arise in the
study of quantum ergodicity. To put our results in a general
context, we follow \cite{Z87} to briefly recall some relevant
notions of the framework of quantum ergodicity.

If $(X,g)$ is an $n$-dimensional compact Riemannian manifold with
Laplace operator $\Delta$, then
$L^2(X)=\oplus_{\lambda_k}\mathcal{H}_{\lambda_k}$, where
$\Delta=-\lambda_k$ on the eigenspaces $\mathcal{H}_{\lambda_k}$ and
$\dim(H_{\lambda_k})<\infty$. We fix ordered orthonormal bases
$\left\{\phi_{k_i}: 1\leq i\leq\dim(H_{\lambda_k})\right\}$ for each
$\mathcal{H}_{\lambda_k}$ to obtain a sequence $\left\{\phi_{k_i}:
k=1,2,3,\ldots, 1\leq i\leq\dim(H_{\lambda_k}) \right\}$ of
orthonormal eigenfunctions. Given a calculus of pseudodifferential
operators on $X$, i.e. an assignment $\Op:
C^{\infty}(S^*X)\rightarrow B(L^2(X))$ of bounded operators $\Op(a)$
to smooth zero order symbols $a$, satisfying the usual requirements
\cite{Zel09}, we associate to a given eigenfunction $\phi_k$ a
distribution $W_k$, called the \emph{Wigner distribution for}
$\phi_k$, defined by $W_k(a):=\langle
\Op(a)\phi_k,\phi_k\rangle_{L^2(X)}$. A distribution
$\mu\in\mathcal{D}'(S^*X)$ is called \emph{weak}$^*$\emph{-limit
point} of the $\left\{W_{k_j}\right\}$ if there is a subsequence
$\mathcal{S}\subseteq\left\{\lambda_{k_i}\right\}$ such that
$\lim_{\mathcal{S}}W_{k_i}(a)=\mu(a)$ for all $a$. One of the
problems in the framework of quantum ergodicity is the question:
What are the weak$^*$-limit points of the $W_{k_i}$? All such limit
distributions are invariant measures for the geodesic flow on
$S^*X$. It is not known which limit points arise and how they depend
on the choice of the $\{\phi_{k_i}\}$.

It was observed in \cite{AZ} that for compact hyperbolic surfaces
$X_{\Gamma}=\Gamma\backslash\mathbb{H}$ Wigner distributions are
asymptotically equivalent (and hence equivalent for the study of
quantum ergodicity) to \emph{Patterson--Sullivan distributions}
$\widehat{PS}_{ir_{k}}$, which are also associated to the sequence
$\left\{\phi_{k}\right\}$ of eigenfunctions. An interesting property
of these Patterson--Sullivan distributions is that they are
invariant under the geodesic flow, so one might hope that the study
of these invariant distributions combined with the relations to
Wigner distributions yield more insight into the questions of
quantum ergodicity for symmetric spaces.

Before we state our results we have to make a few remarks about the
special $\Psi DO$-calculus we use in this paper: S. Zelditch
(\cite{Z84}) introduced a natural quantization for $G/K$, when
$G=PSU(1,1)$, $K=PSO(2)$. It is in fact possible to generalize this
calculus two all rank one symmetric spaces $X:=G/K$, where $G$ is a
connected semisimple Lie group with finite center and $K$ a maximal
compact subgroup of $G$. The basic definitions and properties of
this calculus are given in Section \ref{section PsiDO}. Full details
with all computations concerning this calculus will appear in
\cite{S}. An advantage of this calculus is its $G$-equivariance: Fix
a co-compact and torsion free discrete subgroup $\Gamma$ of $G$ and
let $SX$ denote the unit tangent bundle of $X=G/K$. If $a\in
C^{\infty}(SX)$ is $\Gamma$-invariant (under the natural action of
$G$ on $SX$, see Section \ref{section PsiDO}), then it yields a
pseudodifferential operator on the quotient
$X_{\Gamma}:=\Gamma\backslash G/K$\index{$X_{\Gamma}$, compact
hyperbolic manifold}.

Our setting is as follows: Let $X=G/K$\index{$X=G/K$, symmetric
space} denote a general rank one symmetric space of the noncompact
type, where $G$\index{$G$, semisimple Lie group} is a connected
semisimple Lie group with finite center and $K$\index{$K$, maximal
compact subgroup of $G$} a maximal compact subgroup of $G$. Let
$G=KAN$ be a corresponding Iwasawa decomposition of $G$ and let $M$
denote the centralizer of $A$ in $K$. The geodesic boundary of $X$
can be identified with the flag manifold $B:=K/M$. Let $o:=K\in G/K$
denote the \emph{origin} of the symmetric space $X$. Further, let
$\Delta$, resp. $\Delta_{\Gamma}$, denote the Laplace operator of
$X$, resp. $X_{\Gamma}$. We consider the following automorphic
eigenvalue problem on $X=G/K$:
\begin{eqnarray*}
\Delta \phi &=& - c \, \phi \\
\phi(\gamma z) &=& \phi(z) \textnormal{ for all } \gamma\in\Gamma \textnormal{ and for all } z\in X.
\end{eqnarray*}
In other words, we study the eigenfunctions of the Laplacian on the
compact manifold $X_{\Gamma}=\Gamma\backslash X$. If the
eigenfunctions $\phi$ are real-valued, the eigenvalues $-c\in\rr$,
$-c\leq-\langle\rho,\rho\rangle$, of $\Delta$ are of the form
$-c:=-c_{\lambda}:=-(\langle\lambda,\lambda\rangle+\langle\rho,\rho\rangle)$,
where $\lambda\in\la^*$, the real dual of the Lie algebra $\la$ of
$A$, and where $\langle\,,\,\rangle$ denotes the inner product on
$\la^*$ induced by the Killing form (see Section \ref{Prelim}). We
fix a complete $L^2(X_{\Gamma})$-orthonormal basis
$\left\{\phi_{\lambda_j}\right\}$ of real-valued and
$\Gamma$-invariant eigenfunctions, where the eigenvalues are repeated according to their multiplicity. We hence obtain a corresponding
sequence of eigenvalue parameters $\lambda_j\in\la^*$. Then
\begin{eqnarray*}
\Delta \phi_{\lambda_j} &=& - (\langle\lambda_j,\lambda_j\rangle+\langle\rho,\rho\rangle) \phi_{\lambda_j} \textnormal{ for all } j,\\
\phi_{\lambda_j}(\gamma z) &=& \phi_{\lambda_j}(z) \textnormal{ for all } \gamma\in\Gamma, \, z\in X, \, j\in\nn_0.
\end{eqnarray*}

If $Y$ is a manifold, $u$ a distribution or hyperfunction on $Y$ and
$\phi$ a test function, then we denote the pairing $\langle \varphi,
u\rangle_Y$ by $\int_Y  \varphi(y)u(dy)$.

For each eigenfunction $\phi_{\lambda_j}$ (with exponential growth,
see Section \ref{section boundary values}) of the negative Laplacian
$-\Delta$ with corresponding eigenvalue
$c_j=\langle\lambda_j,\lambda_j\rangle+\langle\rho,\rho\rangle$
there is a unique distribution boundary value (also described in
Section \ref{section boundary values})
$T_{\lambda_j}\in\mathcal{D}'(B)$ such that
\begin{eqnarray*}
\phi_{\lambda_j}(x) = \int_B e^{(i\lambda_j+\rho)\langle x,b\rangle} T_{\lambda_j}(db).
\end{eqnarray*}
Here $\langle x,b\rangle$ denotes the horocyclic bracket defined in
\eqref{horocycle bracket} below. Given $a\in C^{\infty}(SX)$, the
\emph{Wigner distributions} are defined by
\begin{eqnarray*}
W_{\lambda_j,\lambda_k}(a) := \langle \Op(a)\phi_{\lambda_j},
\phi_{\lambda_k}\rangle_{L^2(X_{\Gamma})}.
\end{eqnarray*}
In the special case when $j=k$, we write $W_{\lambda_j}(a) := W_{\lambda_j,\lambda_j}(a)$.

Let $B^{(2)}=(B\times B)\setminus\Delta = \left\{(b,b')\in B\times
B: b\neq b'\right\}$ denote the set of pairs of distinct boundary
points ($\Delta$ denotes the diagonal of $B\times B$). We will
describe the geodesic boundary in Section \ref{Prelim}. Each
geodesic of $X$ has a unique forward limit point and a unique
backward limit point in $B$. In particular, we identify $B^{(2)}$
with the space of geodesics. We will see in Section \ref{section
boundary values} that in the case of $\Gamma$-invariant
eigenfunctions the boundary values $T_{\lambda_j}$ satisfy the
following equivariance property:
\begin{eqnarray}
T_{\lambda_j}(d\gamma b) = e^{-(i\lambda_j+\rho)\langle\gamma o,\gamma b\rangle} T_{\lambda_j}(db), \,\,\,\,\, \gamma\in\Gamma.
\end{eqnarray}
It is then possible to introduce (see Section \ref{PatSul} for
details) functions $d_{\lambda_j}$ on $B^{(2)}$ and a Radon
transform $\mathcal{R}:C_c^{\infty}(SX)\rightarrow
C_c^{\infty}(B^{(2)})$ such that the expression
\begin{eqnarray}\label{expression}
\langle a,PS_{\lambda_j}\rangle_{SX} := \int_{B^{(2)}} \, d_{\lambda_j}(b,b') \, \mathcal{R}(a)(b,b') \, T_{\lambda_j}(db)\, T_{\lambda_j}(db')
\end{eqnarray}
defines a $\Gamma$-invariant distribution on $SX$. We call these
distributions the \emph{Patterson--Sullivan distributions}
associated to the $\left\{\phi_{\lambda_j}\right\}$. The
$PS_{\lambda_j}$ are invariant under the geodesic flow and under
time reversal (see Section \ref{PatSul} for details). The weight
functions $d_{\lambda_j}$ will be called \emph{intermediate values}
because of \eqref{equivariance dlambda}, which generalizes the
intermediate value formula \eqref{intermediate value formula} for
hyperbolic surfaces,

Let $H:KAN\rightarrow\la$ denote the Iwasawa projection (see Section
\ref{Prelim}) and let $w$ denote the non-trivial Weyl group element
(see Section \ref{Prelim}). Given $j\in\nn_0$, define
\begin{eqnarray}
L_{\lambda_j}a(g) := \int_N e^{-(i\lambda_j+\rho)(H(nw))}a(gn)dn, \,\,\,\,\,\,\, a\in C(G),
\end{eqnarray}
whenever the integral exists. 

Following \cite{AZ} we use a cutoff $\chi\in C_c^{\infty}(X)$, which
is a smooth replacement for the characteristic function of a
fundamental domain $\mathcal{F}$ for $\Gamma$ (cf. Section
\ref{PatSul}). A concrete relation between the $W_{\lambda_j}$ and
the $PS_{\lambda_j}$ is given by the operators $L_{\lambda_j}$ and
it generalizes the ``exact formula'' in Theorem 1.1 of \cite{AZ}:

\begin{thm}\label{Intertwining Formula diagonal}
Let $a\in C^{\infty}(SX_{\Gamma})$. Then
\begin{eqnarray}
\langle \Op(a)\phi_{\lambda_j},
\phi_{\lambda_j}\rangle_{L^2(X_{\Gamma})} = \langle
L_{\lambda_j}(\chi a), PS_{\lambda_j}\rangle_{SX}.
\end{eqnarray}
\end{thm}

Still following \cite{AZ}, we also define normalized
Patterson--Sullivan distributions
\begin{eqnarray}
\widehat{PS}_{\lambda_j} := \frac{1}{\langle 1,PS_{\lambda_j}\rangle_{SX_{\Gamma}}} PS_{\lambda_j},
\end{eqnarray}
which satisfy the same normalization condition $\langle
1,\widehat{PS}_{\lambda_j}\rangle_{SX_{\Gamma}}=1$ as the
$W_{\lambda_j}$ on the quotient $SX_{\Gamma}$.

As was pointed out in the introduction of \cite{AZ} it is of
interest to also have analogous results for off-diagonal matrix
entries. To this end we introduce (in Section \ref{section
off-diag}) off-diagonal Patterson--Sullivan distributions
$PS_{\lambda_j,\lambda_k}$ such that
$PS_{\lambda_j,\lambda_j}=PS_{\lambda_j}$ for all $j\in\nn_0$. We
then prove the off-diagonal analog of Theorem \ref{Intertwining
Formula diagonal}:

\begin{thm}\label{Intertwining Formula}
Let $a\in C^{\infty}(SX_{\Gamma})$. Then
\begin{eqnarray}
\langle \Op(a)\phi_{\lambda_j},
\phi_{\lambda_k}\rangle_{L^2(X_{\Gamma})} = \langle
L_{\lambda_k}(\chi a), PS_{\lambda_j,\lambda_k}\rangle.
\end{eqnarray}
\end{thm}

Theorem \ref{Intertwining Formula diagonal} is an immediate
consequence of Theorem \ref{Intertwining Formula}, but we
intentionally separated the definitions and the results. One reason
is that the definitions are based on quite different ideas and that
the $PS_{\lambda_j}$ have nicer invariance properties than the
$PS_{\lambda_j,\lambda_k}$.
Another reason is that the normalization of the $PS_{\lambda_j}$ motivates the normalization of the $PS_{\lambda_j,\lambda_k}$ (see Definition \ref{normalize off-diag}).

Finally, we generalize the ``asymptotic formula'' in Theorem 1.1 of
\cite{AZ} to off-diagonal elements:

\begin{thm}\label{Asymptotic}
Let $a\in C^{\infty}(SX_{\Gamma})$. Assume that $\lambda_{j_n}$,
$\lambda_{k_n}\rightarrow\infty$ are sequences of spectral
parameters such that $|\lambda_{j_n} - \lambda_{k_n}| \leq \tau$ for
some $\tau>0$. Then
\begin{eqnarray*}
\langle \Op(a)\phi_{\lambda_{j_n}},
\phi_{\lambda_{k_n}}\rangle_{L^2(X_{\Gamma})} = \langle
a,\widehat{PS}_{\lambda_j,\lambda_k}\rangle_{SX_{\Gamma}} +
O(1/|\lambda_{k_n}|) \,\,\,\,\,\,
(n\rightarrow\infty).
\end{eqnarray*}
\end{thm}

We thank S. Hansen, J. M\"{o}llers, A. Pasquale and in particular N.
Anantharaman and S. Zelditch for helpful discussions. The comparison
of our results with the ones from their draft \cite{AZ2} made us
realize a small error in our original formulation of Theorem
\ref{Asymptotic}. Special thanks go to M. Olbrich for providing the
idea of a proof of Proposition \ref{Proposition Olbrich}.

\section{Preliminaries}\label{Prelim}

In this section we collect a number of geometric definitions and
facts needed to formulate our main results.

\subsection*{Semisimple Lie Groups}
Let $G$ be a non-compact connected semisimple Lie group with finite
center, $\g$ the Lie algebra of $G$, and $\langle\,,\,\rangle$ the
\emph{Killing form}\index{$\langle\,,\,\rangle$, Killing form} of
$\g$. Let $\theta$\index{$\theta$, Cartan involution} be a
\emph{Cartan involution} of $\g$ such that the form
$(X,Y)\mapsto-\langle X,\theta Y\rangle$ is positive definite on
$\g\times\g$. Let $\g=\lk + \lp$ be the decomposition of $\g$ into
eigenspaces of $\theta$ and $K$ the analytic subgroup of $G$ with
Lie algebra $\lk$. We choose a maximal abelian subspace $\la$ of
$\lp$ and denote by $\la^*$ its dual and $\la^*_{\cc}$ the
complexification of $\la^*$. At this point we do {\it not yet} make
the assumption that the rank of $X=G/K$, i.e. $\dim \la$, is one.
Later, however, it will be indispensable. Let $A=\exp{\la}$ denote
the corresponding analytic subgroup of $G$ and let $\log$ denote the
inverse of the map $\exp:\la\rightarrow A$.

Given $\lambda\in\la^*$, put $\g_{\lambda} =
\left\{X\in\g:[H,X]=\lambda(H)X\,\,\forall\,H\in\la\right\}$. If
$\lambda\neq0$ and $\g_{\lambda}\neq\left\{0\right\}$, then
$\lambda$ is called a \emph{(restricted) root} and
$m_{\lambda}=\dim(\g_{\lambda})$ is called its \emph{multiplicity}.
Let $\g_{\cc}$ denote the complexification of $\g$ and if
$\mathfrak{s}$ is any subspace of $\g$ let $\mathfrak{s}_{\cc}$
denote the complex subspace of $\g_{\cc}$ spanned by $\mathfrak{s}$.

For $\lambda\in\la^*$ let $H_{\lambda}\in\la$ be determined by
$\lambda(H)=\langle H_{\lambda},H\rangle$ for all $H\in\la$. For
$\lambda,\mu\in\la^*$ we put $\langle\lambda,\mu\rangle:=\langle
H_{\lambda},H_{\mu}\rangle$. Since $\langle\,,\,\rangle$ is positive
definite on $\lp\times\lp$ we put
$|\lambda|:=\langle\lambda,\lambda\rangle^{1/2}$ for
$\lambda\in\la^*$ and $|X|:=\langle X,X\rangle^{1/2}$ for $X\in\lp$.
The $\cc$-bilinear extension of $\langle\,,\,\rangle$ to
$\la_{\cc}^*$ will be denoted by the same symbol.

Let $\la'$ be the open subset of $\la$ where all restricted roots
are $\neq0$. The components of $\la'$ are called Weyl chambers. We
fix a Weyl chamber $\la^+$ and call a root $\alpha$ positive ($>0$)
if it is positive on $\la^+$. Let $\la^*_+$ denote the corresponding
Weyl chamber in $\la^*$, that is the preimage of $\la^+$ under the
mapping $\lambda\mapsto H_{\lambda}$. Let $\Sigma$ denote the set of
restricted roots, $\Sigma^+$ the set of positive roots and
$\Sigma^-:= -\Sigma^+$ the set of negative roots.

Let $\Sigma_0=\left\{ \alpha\in\Sigma: \frac{1}{2}\alpha\notin\Sigma
\right\}$, and put $\Sigma_0^+=\Sigma^+\cap\Sigma_0$,
$\Sigma_0^-=\Sigma^-\cap\Sigma_0$. We set
$\rho:=2^{-1}\Sigma_{\alpha\in\Sigma^+}m_{\alpha}\alpha$\index{$\rho=2^{-1}\Sigma_{\alpha\in\Sigma^+}m_{\alpha}\alpha$}
and let $N$ denote the analytic subgroup of $G$ with Lie algebra
$\lnn:=\Sigma_{\alpha>0}\g_{\alpha}$. Then
$\overline{\lnn}=\theta(\lnn)=\Sigma_{\alpha<0}\g_{\alpha} $. The
involutive automorphism $\theta$ of $\mathfrak{g}$ extends to an
analytic involutive automorphism of $G$, also denoted by $\theta$,
whose differential at the identity $e\in G$ is the original
$\theta$. It thus makes sense to define $\overline{N}=\theta N$. The
Lie algebra of $\overline{N}$ is $\theta(\lnn)$.

Let $M$ denote the centralizer of $A$ in $K$ and let $M'$ denote the
normalizer of $A$ in $K$. Let $W$ denote the (finite) \emph{Weyl
group} $M'/M$\index{$W$, Weyl group}. The group $W$ acts as a group
of linear transformations of $\la$ and also on $\la_{\cc}^*$ by
$(s\lambda)(H):=\lambda(s^{-1}H)$ for $s\in W$, $H\in\la$ and
$\lambda\in\la_{\cc}^*$. Let $r$ denote the order of $W$ and let
$w_1,\ldots,w_r$ be a complete set of representatives in $M'$. Let
$A^+:=\exp(\la^+)$, $B:=K/M$, $P:=MAN$. Then we have the
decompositions
\begin{itemize}
\item[(1)] $G=KAN$ \,\,\,\,\,\,\,\,\,\,\,\,\,\,\,\,\,\,\, (Iwasawa decomposition),
\item[(2)] $G=\bigcup_{j=1}^{r}Pw_jP$ \,\,\,\hspace{0.3mm} (Bruhat decomposition).
\end{itemize}
Here (1) means that each $g\in G$ can be uniquely written in the
form
\begin{eqnarray}
g=k(g)\exp H(g) n(g),
\end{eqnarray}
where $k(g)\in K$, $H(g)\in\la$, $n(g)\in N$. The functions $k,H,n$
are called \emph{Iwasawa projections}. In (2), the union is
disjoint. Let $w^*$ denote the Weyl group element mapping $\la^+$ to
$-\la^+$. Exactly one of the summands in (2), namely $Pw^*P$, is
open in $G$. Thus the set $\overline{N}MAN$ is open in $G$.

We call $\dim(\la)$ the real rank of $G$ and the rank of the
symmetric space $X=G/K$. Let $o:=K\in G/K$ denote the \emph{origin}
of $X$. If $X$ has rank one, the Weyl group has only two elements.
In this case we denote the nontrivial Weyl group element by $w$ and
pick an element $m'\in M'$ such that  $m'M=w\in W$. By abuse of
notation we write $w$ for $m'$. Then we have the important formula
\begin{eqnarray}\label{w and a}
waw^{-1}=a^{-1} \,\,\,\, \forall a\in A.
\end{eqnarray}
If $Y$ is a manifold satisfying the second countability axiom we
write $\mathcal{D}(Y)$ for the space of $C^{\infty}$ functions on
$Y$ of compact support. $\mathcal{D}'(Y)$ denotes the dual space of
distributions on $Y$. The space $\mathcal{E}(Y)$ denotes the space
of $C^{\infty}$ functions on $Y$ and $\mathcal{E}'(Y)$ denotes the
dual space of distributions on $Y$ of compact support.

\subsection*{Normalization of Measures}\label{Measure theoretic preliminaries}
We briefly recall some normalizations of the measures on the
homogeneous spaces we work with. We follow \cite{He00}. The Killing
form induces euclidean measures on $A$, $\la$ and $\la^*$. For
$l=\dim(A)$ we multiply these measures by $(2\pi)^{-l/2}$ and obtain
invariant measures $da, dH$ and $d\lambda$ on $A,\mathfrak{a}$ and
$\mathfrak{a}^*$. This normalization has the advantage that the
euclidean Fourier transform of $A$ is inverted without a
multiplicative constant. We normalize the Haar measures $dk$ and
$dm$ on the compact groups $K$ and $M$ such that the total measure
is $1$. If $U$ is a Lie group and $P$ a closed subgroup, with left
invariant measures $du$ and $dp$, the $U$-invariant measure
$du_P=d(uP)$ on $U/P$ (if it exists) will be normalized by
\begin{eqnarray}\label{integral quotient}
\int_U f(u)du=\int_{U/P}\left(\int_P f(up)dp\right)du_P.
\end{eqnarray}
This measure exists in particular if $P$ is a compact subgroup of
$U$. In particular, we have a $K$-invariant measure $dk_M=d(kM)$ on
$K/M$ of total measure $1$. We also have a $G$-invariant measure
$dx=dg_K=d(gK)$ on  $X=G/K$. By uniqueness, $dx$ is a constant
multiple of the measure on $X$ induced by the Riemannian structure
on $X$ given by the Killing form. The Haar measures $dn$ and
$d\overline{n}$ on the nilpotent groups $N$ and $\overline{N}$ are
normalized such that
\begin{eqnarray}\label{theta(dn)}
\theta(dn)=d\overline{n}, \hspace{7mm} \int_{\overline{N}}e^{-2\rho(H(\overline{n}))}d\overline{n}=1.
\end{eqnarray}

The Haar measure on $G$ can (\cite{He00}, Ch. I, \S5) then be normalized such that
\begin{eqnarray}\label{integral formula G}
\int_G f(g)dg &=& \int_{KAN}f(kan)e^{2\rho(\log a)} \, dk \, da \, dn \\
&=& \int_{NAK}f(nak)e^{-2\rho(\log a)} \, dn \, da \, dk
\end{eqnarray}
for all $f\in C_c(G)$. Let $f_1\in C_c(AN)$, $f_2\in C_c(G)$, $a\in
A$. Then (\cite{He00}, pp. 182)
\begin{eqnarray} \int_N f_1(na) \, dn = e^{2\rho(\log(a))}\int_N f_1(an) \, dn \end{eqnarray}
and
\begin{eqnarray}\label{follows that} \int_G f_2(g) \, dg = \int_{KNA}f_2(kna) \, dk \, dn \, da
= \int_{ANK}f_2(ank) \, da \, dn \, dk.\end{eqnarray}
Let $f_3\in C_c(X)$. It follows from \eqref{follows that} that
\begin{eqnarray}\label{integral formula AN and G/K}
\int_X f_3(x) \, dx = \int_{AN}f_3(an\cdot o) \, da \, dn.
\end{eqnarray}

For any (restricted) root $\alpha$ we write $\alpha_0:=\alpha/\langle\alpha,\alpha\rangle$. We will need \emph{Harish-Chandra's} $e$\emph{-functions} (\cite{He94}, p. 163, the rank of $X$ is arbitrary)
\begin{eqnarray}\label{e function}
e_s(\lambda) = \prod_{\alpha\in\Sigma_s^+} \Gamma\left(\frac{m_{\alpha}}{4}+\frac{1}{2}+\frac{\langle i\lambda,\alpha_0\rangle}{2}\right) \Gamma\left(\frac{m_{\alpha}}{4}+\frac{m_{2\alpha}}{2}+\frac{\langle i\lambda,\alpha_0\rangle}{2}\right),
\end{eqnarray}
where $s\in W$, $\Sigma_s^+=\Sigma_0^+\cap s^{-1}\Sigma_0^-$ and where $\Gamma$ denotes the classical Gamma-function. Let $X$ have rank one. Now, the set $\Sigma$ of (restricted) roots contains at most two positive elements: $\alpha$ and possibly $2\alpha$. We adopt the usual convention that $m_{2\alpha}=0$ if $2\alpha$ is not a root. \emph{Harish-Chandra's} $c$\emph{-function} (\cite{He00}, Ch. IV, \S6) is the meromorphic function
\begin{eqnarray}\label{c function}
c(\lambda) \, = \, c_0 \, \frac{2^{-\langle i\lambda,\alpha_0\rangle} \, \Gamma\left(\langle i\lambda,\alpha_0\rangle\right)}{\Gamma\left(\frac{m_{\alpha}}{4}+\frac{1}{2}+\frac{\langle i\lambda,\alpha_0\rangle}{2}\right) \, \Gamma\left(\frac{m_{\alpha}}{4}+\frac{m_{2\alpha}}{2}+\frac{\langle i\lambda,\alpha_0\rangle}{2}\right)}, \,\,\,\,\,\, \lambda\in\la^{*}_{\cc},
\end{eqnarray}
where $c_0 = 2^{\frac{1}{2}m_{\alpha}+m_{2\alpha}} \, \Gamma(\frac{m_{\alpha}}{2}+\frac{m_{2\alpha}}{2}+\frac{1}{2})$.

\subsection*{Geodesics, Boundary and the Unit Tangent Bundle}
$G$ acts on $G/P$ via $g\cdot xP=gxP$ and $K/M\rightarrow G/P, \,
kM\mapsto kP$ is a diffeomorphism (\cite{He01}, p. 407) inverted by
$gP\mapsto k(g)M$, where $g=k(g)\exp{H(g)}n(g)$. Hence this map
intertwines the $G$-action on $G/P$ with the action on $K/M$ defined
by $g\cdot kM=k(gk)M$. These spaces are thus equivalent for the
study of $B=K/M=G/P$. Although the following remarks are basically
trivial, we write them down for later reference: With respect to the
actions described above, the stabilizer of $M\in K/M$ is the
subgroup $P=MAN$. The action of the groups $AN$ and $P$ on $G/K$ are
transitive. For the remainder of this section, let $X=G/K$ be of
rank one. As above, let $w\in M'$ denote a representative of the
nontrivial Weyl group element.
\begin{lem}
$P=MAN$ acts transitively on $G/P\setminus\left\{P\right\}$.
\end{lem}
\begin{proof}
This follows from the Bruhat decomposition
\begin{eqnarray*}
G = P \cup PwP \,\,\,\,\,\, \textnormal{(disjoint union)}.
\end{eqnarray*}
In fact, let $gP\in G/P\setminus\left\{P\right\}$. Then $g\notin P$,
so $g$ is of the form $p_1wp_2$, where $p_1,p_2\in P$. Hence
$gP=p_1wP=p_1\cdot wP$ and we have proven that each $gP\neq P\in
G/P$ lies in the $P$-orbit of $wP$.
\end{proof}

\begin{rem}\label{unit tangent bundle}
Let $H_0$ denote the unique unit vector (with respect to the norm on
$\la$ induced by the Killing form) in $\la^+$. It is well known
(\cite{He00}) that $K\cdot H_0=S(\mathfrak{p})$, i.e., the group $K$
acts transitively on the set $S(\mathfrak{p})$ of unit vectors in
the tangent space $T_o X = \mathfrak{p}$. The subgroup $AN=NA$ of
$G$ acts transitively on $G/K$, so $G=NAK$ acts transitively on the
the unit tangent bundle $SX$ of $X=G/K$. The group $M$ is the
stabilizer in $K$ of $H_0\in S(\mathfrak{p})$. Hence the unit
tangent bundle of $X$ can be identified $G$-equivariantly with the
homogeneous space $G/M$. We will from now on write $SX=G/M$ (for $X$
of real rank one). The geodesic flow on $G/M$ reads as the action of
$A$ by right translations on $G/M$.
\end{rem}

As in the introduction, consider the space $B\times B$ and its
diagonal $\Delta$. We let $B^{(2)}=(B\times B)\setminus\Delta$
denote the set of distinct boundary points. We may now describe the
space of geodesics and the geodesic connections in the rank one
case. We describe the map that assigns to a geodesic its forward and
backward limit points.

We call $\gamma_{H_0}(t)=e^{t{H_0}}\cdot o$ the \emph{standard
geodesic}\index{standard geodesic}. If we write $B=K/M$, the forward
limit point $b_{\infty}$ of the standard geodesic identifies with
$M\in K/M$ (that is $P\in G/P$) and (since $\Ad_G(w)$ is $-\id$ on
$\la$) its backward limit point $b_{-\infty}$ identifies with $wM\in
K/M$ (that is $wP\in G/P$). Since $wM\neq M$ in $K/M$, the point
$(M,wM)$ is an element of $B^{(2)}$ and the standard geodesic is the
unique (up to parameter translation and time reversal) geodesic of
$X$ that joins the boundary points $M$ and $wM$ at infinity. We also
write $a_t:=e^{tH_0}\in A$. Then the standard geodesic is the curve
$t\mapsto a_t \cdot o$.

We consider the action of $G$ on $B^{(2)}$ given by
\begin{eqnarray}\label{G on B^{(2)}}
G \times B^{(2)} \rightarrow B^{(2)}, \hspace{3mm} g\cdot (b_1,b_2) = (g\cdot b_1, g\cdot b_2).
\end{eqnarray}

\begin{lem}\label{space of geodesics}
$G$ acts transitively on $B^{(2)}$. The stabilizer of
$(b_{\infty},b_{-\infty})\in B^{(2)}$ is the subgroup $MA$ of $G$.
\end{lem}

\begin{prop}
$B^{(2)}=G/MA$ as homogeneous spaces.
\end{prop}
\begin{proof}
Let $b_1 \neq b_2$ be points in $B$. Since $K$ acts transitively on
$B$, we find a $k\in K$ such that $k\cdot b_1 = b_{\infty}$. Since
$P$ acts transitively on $B\setminus\left\{b_{\infty}\right\}$, we
also find a $p\in P$ such that $p\cdot k\cdot b_2=b_{-\infty}$. Let
$g=pk$. Then $g\cdot (b_1,b_2) = (b_{\infty},b_{-\infty})$.

It remains to show
$g\cdot(b_{\infty},b_{-\infty})=(b_{\infty},b_{-\infty})$ if and
only if $g\in MA$. Let $ma\in MA$. Then $ma\cdot
b_{\infty}=b_{\infty}$ and $ma \cdot b_{-\infty} = m\cdot a\cdot wP
= w\tilde{m}\tilde{a}P = wP$, since $M'$ normalizes $A$ and $M$.
Hence $MA$ acts trivially on $(b_{\infty},b_{-\infty})\in B^{(2)}$.
For the converse assume
$g\cdot(b_{\infty},b_{-\infty})=(b_{\infty},b_{-\infty})$. Then
$g\cdot b_{\infty}$, so $g = man \in MAN$.

It suffices to prove $n=e$. For $b\in B$ let $G_b$ denote the
subgroup of $G$ fixing $b$. Then $n\in G_{b_{\infty}}\cap
G_{b_{-\infty}} = MAN \cap wMANw^{-1}$, so $n\in wMANw^{-1} = MA
\cdot wNw^{-1} = MA \overline{N}$. Hence there exists an element
$\overline{n}'\in\overline{N}$ such that $n\overline{n}'=ma \in
N\overline{N}\cap MA = \left\{e\right\}$ (cf. \cite{He01}, Ch. VI,
Exercise B2. See also \cite{He94}, Lemma 1.6 on page 79.). But
$N\cap\overline{N}=\left\{e\right\}$, since $\mathfrak{g}$ is the
direct vector space sum of the root-spaces $\mathfrak{g}_{\alpha}$.
Hence $n=e$ as desired,
\end{proof}

\begin{defn}\label{definition g(b,b')}
We will from now on always write $g(b,b')MA\in G/MA$ for the unique
coset corresponding to $(b,b')\in B^{(2)}$. The representative
$g(b,b')$ is uniquely determined modulo $MA$.
\end{defn}

Remark \ref{unit tangent bundle} yields a $G$-equivariant
identification $G/M=SX$. Another identification is $G/M=X\times B$:
It is clear that with respect to the diagonal action of $G$ on
$X\times B$ the group $M$ is the stabilizer of $(K,P)\in G/K\times
G/P$. Using the Iwasawa decomposition (see \cite{S} for details) we
also see that $G$ acts transitively on the space $X\times B$. This
induces the following $G$-equivariant identification $SX=X\times B$:
If $(z,b)\in X\times B$, then let $v(z,b)$ denote the unit vector in
$SX$ tangential to the geodesic through $z$ with forward endpoint
$b$. This geodesic exists since $X$ has rank one (see \cite{E96},
\cite{Q} and also \cite{S} for details).

\subsection*{Horocycle bracket and Iwasawa Projection}\label{An equivariance property}

In this subsection, we describe the so-called \emph{horocycle
bracket} on $X\times B$, because we need some formulae corresponding
to this inner product. For details on the geometric interpretation
of this horocycle bracket see \cite{He94}, Ch. II (and \cite{S}).

Let $\langle\cdot,\cdot\rangle:X\times B \rightarrow \mathfrak{a},
(x,b)\mapsto\left\langle x,b\right\rangle$ be defined by
\begin{eqnarray}\label{horocycle bracket}
\left\langle x,b\right\rangle = \left\langle gK,kM\right\rangle = -H(g^{-1}k).
\end{eqnarray}
Each $(x,b)$ is of the form $(gK,kM)$ and it is easy to see that
\eqref{horocycle bracket} is well-defined. We remark that the use of
$\langle\,,\,\rangle$, whether we mean the Killing form or the
horocycle bracket, will always be clear from the context.

\begin{prop}
$\left\langle \cdot ,\cdot \right\rangle$ is invariant under the diagonal action of $K$ on $X\times B$.
\end{prop}

Recall that $g\in G$ acts on $K/M$ by $g\cdot kM=k(gk)M$ (Iwasawa projection).
\begin{lem}\label{invariance0}
Let $g_1,g_2\in G$, $k\in K$. Then $H(g_1 g_2 k)=H(g_1 k(g_2 k)) + H(g_2 k)$.
\end{lem}

\begin{proof}
Decompose $g_2k=\tilde{k}\tilde{a}\tilde{n}$ and
$g_1\tilde{k}=k'a'n'$. Then $$H(g_1g_2k) =
H(k'a'n'\tilde{a}\tilde{n}) = H(a'n'\tilde{a}).$$ Since $A$
normalizes $N$ this equals $\log(a') + \log(\tilde{a})$.
\end{proof}

\begin{lem}
Let $x=hK\in G/K$, $b=kM\in K/M$, $g\in G$. Then
\begin{eqnarray}\label{equivariance}
\left\langle g\cdot x,g\cdot b \right\rangle
= \left\langle x,b \right\rangle + \left\langle g\cdot o,g\cdot b \right\rangle.
\end{eqnarray}
\end{lem}

\begin{proof}
By definition, $\left\langle g\cdot x,g\cdot b \right\rangle =
-H(h^{-1}g^{-1}k(gk))$. Then by Lemma \ref{invariance0} applied to
$g_1=h^{-1}g^{-1}$ and $g_2=g$ this equals
\begin{eqnarray*}
-H(h^{-1}g^{-1}gk) + H(gk) = -H(h^{-1}k) + H(gk).
\end{eqnarray*}
For $h=e$ we obtain $\langle g\cdot o, g\cdot b\rangle = -H(k) + H(gk) = H(gk)$. Hence
\begin{eqnarray*}
\left\langle g\cdot x,g\cdot b \right\rangle - \left\langle g\cdot o,g\cdot b \right\rangle
&=& [-H(h^{-1}k) + H(gk)] - [-H(k) + H(gk)],
\end{eqnarray*}
which equals $-H(h^{-1}k) = \langle hK, kM\rangle = \langle x,b\rangle$.
\end{proof}

\begin{lem}\label{facts}
Let $\gamma, g\in G$ and $w\in K$. Then
\begin{itemize}
\item[(i)] $\langle g\cdot o, g\cdot M\rangle = H(g)$
and $\langle g\cdot o, g\cdot wM\rangle = H(gw)$.
\item[(ii)] $H(\gamma g)
= H(g) + \langle\gamma\cdot o,\gamma g\cdot M\rangle$ and $H(\gamma
gw) = H(gw) + \langle\gamma\cdot o,\gamma g\cdot wM\rangle$.
\end{itemize}
\end{lem}

\begin{proof}
\emph{(i)} is a direct computation. The second part of \emph{(ii)}
follows from the first part applied to $gw$ instead of $g$. For this
assertion, let $z=g\cdot o$. Then by \emph{(i)}
\begin{eqnarray*}
H(\gamma g) = \langle \gamma g \cdot o, \gamma g \cdot M\rangle
= \langle\gamma\cdot z,\gamma g\cdot M\rangle
\end{eqnarray*}
and by \eqref{equivariance} this equals $\langle z,g\cdot M\rangle +
\langle \gamma\cdot o, \gamma g\cdot M\rangle = H(g) +
\langle\gamma\cdot o,\gamma g\cdot M\rangle$.
\end{proof}

\section{Helgason Boundary Values}\label{section boundary values}

In this section we recall the Poisson transform, which plays a key
role in the proofs of our results, and use it to prove the estimate
\eqref{BtimesB} which will allow us to define the
Patterson--Sullivan distributions. Even though part of what we
describe here could be done in greater generality we restrict
ourselves to the case of rank one spaces.

\subsection*{Eigenfunctions and Poisson Transform}

We fix a co-compact, torsion free discrete subgroup $\Gamma$ of $G$
and choose a $G$-invariant measure $\nu$ on $\Gamma\backslash G$
such that
\begin{eqnarray*}
\int_G f(x) dx = \int_{\Gamma\backslash G}\left(\sum_{\gamma}f(\gamma x)\right)d\nu(\Gamma x)
\end{eqnarray*}
for $f\in C_c(G)$. We will denote the Hilbert space
$L^2(\Gamma\backslash G,\nu)$ simply by $L^2(\Gamma\backslash G)$.
The $G$-invariance of $\nu$ implies that the equation
\begin{eqnarray*}
(R_{\Gamma}(g)f)(\Gamma x) = f(\Gamma xg)
\end{eqnarray*}
($g,x\in G$, $f\in L^2(\Gamma\backslash G)$) defines a unitary
representation $R_{\Gamma}$ of $G$ on $L^2(\Gamma\backslash G)$,
which is called the \emph{right-regular representation} of $G$ on
$\Gamma\backslash G$.

As before, let $\Delta$ denote the Laplace operator of $X$. The
eigenspaces corresponding to eigenvalues $-c \leq
-\langle\rho,\rho\rangle$ of $\Delta$ are (\cite{He74}, Theorem 7.1)
the spaces
\begin{eqnarray*}
\mathcal{E}_{\lambda}(X) = \left\{f\in\mathcal{E}(X):
\Delta f=-(\langle\lambda,\lambda\rangle+\langle\rho,\rho\rangle)f \right\},
\end{eqnarray*}
where $\lambda\in\la^*$ and where $\langle\,,\,\rangle$ denotes the
inner product on $\la^*$ induced by the Killing form as described in
Section \ref{Prelim}. We fix a $\Gamma$-invariant eigenfunction
$\phi\in\mathcal{E}_{\lambda}(X)$ and assume that $\phi$ is
normalized with respect to the $L^2(X_{\Gamma})$-norm. Then
$\Delta\phi=-(\langle\lambda,\lambda\rangle+\langle\rho,\rho\rangle)\phi$.

Let $\mathcal{A}(B)$ denote the vector space of analytic functions
on $B=K/M$, topologized as in \cite{He74}, Section 5. The
\emph{analytic functionals} are (loc. cit.) the functionals in the
dual space $\mathcal{A}'(B)$ of $\mathcal{A}(B)$. Fix
$\lambda\in\mathfrak{a}^*$ and recall the following fundamental
result (\cite{He94}, p. 507):

\begin{thm}\label{Helgason boundary values}
The Poisson--Helgason transform $P_{\lambda}:
\mathcal{A}'(B)\rightarrow\mathcal{E}_{\lambda}(X)$ given by
\begin{eqnarray}\label{Poisson--Helgason transform}
P_{\lambda}(T)(z) := \int_B e^{(i\lambda+\rho)\langle z,b\rangle} T(db)
\end{eqnarray}
is a bijection of the dual space $\mathcal{A}'(B)$ onto the
eigenspace $\mathcal{E}_{\lambda}(X)$.
\end{thm}

For an eigenfunction $f\in \mathcal{E}_{\lambda}(X)$ of the
Laplacian we call the unique functional $T_{f}$ with
$f=P_{\lambda}(T_{f})$, given by Theorem \ref{Helgason boundary
values}, the \emph{boundary values} of $f$. We will now consider a
special class of these eigenfunctions that have distributional
boundary values: Let $d_X$ denote the distance function on $X$ and
define the space $\mathcal{E}^*(X)$ of smooth functions of
exponential growth by
\begin{eqnarray}\label{exponential growth}
\mathcal{E}^*(X) := \left\{    f\in\mathcal{E}(X)\mid \exists C>0:
|f(x)|\leq C e^{C d_X(o,x)} \,\,\, \forall x\in X  \right\}.
\end{eqnarray}
We put $\mathcal{E}_{\lambda}^*(X) := \mathcal{E}^*(X)\cap
\mathcal{E}_{\lambda}(X)$ and recall \eqref{e function}. Then (cf.
\cite{He94}, p. 508):

\begin{thm}\label{distribution}
Let $\lambda\in\mathfrak{a}^*_{\cc}$ be such that $e_w(\lambda)\neq
0$. Then $P_{\lambda}(\mathcal{D}'(B))=\mathcal{E}_{\lambda}^*(X)$.
\end{thm}

$G$ acts on $B$, hence on $\mathcal{D}'(B)$ by push-forward: Given
$T\in\mathcal{D}'(B)$, a test function $f\in\mathcal{E}(B)$ and
$g\in G$, the action is $(gT)(f)=T(f\circ g^{-1})$. When we denote
the pairing between distributions and test functions by an integral,
we also write $T(d\gamma b)$ for $(\gamma T)(db)$. Consider a
$\Gamma$-invariant eigenfunction $\phi$ with boundary values
$T_{\phi}$: Then $\phi(\gamma z)=\phi(z)$ for all $\gamma$ and $z$
implies (recall $\left\langle g\cdot x,g\cdot b \right\rangle =
\left\langle x,b \right\rangle + \left\langle g\cdot o,g\cdot b
\right\rangle$ from equation \eqref{equivariance})
\begin{eqnarray*}
\phi(z)
&=& \int_B e^{(i\lambda+\rho)\langle \gamma z ,b\rangle} T_{\phi}(db)
 = \int_B e^{(i\lambda+\rho)\langle \gamma z,\gamma b\rangle} T_{\phi}(d\gamma b) \\
&=& \int_B e^{(i\lambda+\rho)\langle z,b\rangle} e^{(i\lambda+\rho)\langle\gamma o,\gamma b\rangle}
     T_{\phi}(d\gamma b).
\end{eqnarray*}
By uniqueness of the Poisson--Helgason transform (Theorem
\ref{Helgason boundary values}) we obtain
\begin{eqnarray}\label{boundary values equivariance 1}
T_{\phi}(d\gamma b) = e^{-(i\lambda+\rho)\langle\gamma o,\gamma b\rangle} T_{\phi}(db).
\end{eqnarray}

\subsection*{Spherical Principal Series}

We recall some facts concerning the \emph{principal
series}\index{principal series} representations of $G$. Following
\cite{He94} and \cite{Wil}, let $\lambda\in\la$ and consider the
representation $\sigma_{\lambda}(man)=e^{(i\lambda+\rho)\log(a)}$ of
$P=MAN$ on $\cc$. We denote the \emph{induced
representation}\label{induced representation} on $G$ by
$\pi_{\lambda}=\Ind_P^G(\sigma_{\lambda})$. The \emph{induced
picture}\label{induced picture} of this representation is
constructed as follows: A dense subspace of the representation space
is
\begin{eqnarray*}
H_{\lambda}^{\infty} := \left\{ f\in C^{\infty}(G): f(gman)=e^{-(i\lambda+\rho)\log(a)}f(g)\right\}
\end{eqnarray*}
with inner product
\begin{eqnarray*}
(f_1,f_2) = \int_{K/M}f_1(k)\overline{f_2(k)}\,dk = \langle {f_1}_{|K},{f_2}_{|K}\rangle_{L^2(K/M)}
\end{eqnarray*}
and corresponding norm $\|f\|^2 = \int_{K/M}|f(k)|^2\,dk$. The group
action of $G$ is given by $(\pi_{\lambda}(g)f)(x)=f(g^{-1}x)$. The
actual Hilbert space, which we denote by $H_{\lambda}$, and the
representation on $H_{\lambda}$, which we also denote by
$\pi_{\lambda}$, is obtained by completion (cf. \cite{Wil}, Ch. 9).
The representations $\pi_{\lambda}$ ($\lambda\in\la$) form the
\emph{spherical principal series} of $G$.
$(\pi_{\lambda},H_{\lambda})$ is a unitary (\cite{He94}, p. 528) and
irreducible (loc. cit. p. 530) Hilbert space representation.

Given $f\in C^{\infty}(K/M)$ we may extend it to a function on $G$
by $\tilde{f}(g)=e^{-(i\lambda+\rho)H(g)}f(k(g))$. A direct
computation shows that $\tilde{f}\in H^{\infty}_{\lambda}$. On the
other hand, if $f\in H^{\infty}_{\lambda}$, then the restriction
$f_{|K}$ of $f$ to $K$ is an element of $C^{\infty}(K/M)$. Moreover,
if $f\in C^{\infty}(K/M)$ and if $\tilde{f}$ is as above, then
$\tilde{f}_{|K}=f$. The mapping $f\mapsto\tilde{f}$ described above
is isometric with respect to the $L^2(K/M)$-norm. We may hence
identify $C^{\infty}(K/M)\cong H_{\lambda}^{\infty}$. The advantage
is that the representation space is independent of $\lambda$. The
group action on $C^{\infty}(K/M)$ is realized by
\begin{eqnarray}\label{compact model action}
(\pi_{\lambda}(g)f)(kM) = f(k(g^{-1}k)M) e^{-(i\lambda+\rho)H(g^{-1}k)}.
\end{eqnarray}
This is called the \emph{compact picture} of the (spherical)
principal series. Notice that for $g\in K$ the group action
\eqref{compact model action} simplifies to the left-regular
representation of the compact group $K$ on $K/M$.

Let $\lambda\in\la$. It follows from
\begin{eqnarray}\label{Poisson follows}
(\pi_{\lambda}(g)1)(k)=e^{-(i\lambda+\rho)H(g^{-1}k)}=e^{(i\lambda+\rho)\langle gK,kM\rangle}
\end{eqnarray}
that the Poisson transform $P_{\lambda}(T): G/K\rightarrow \cc$ of
$T\in\mathcal{D}'(B)$ is given by
\begin{eqnarray}\label{Poisson intertwines}
P_{\lambda}(T)(gK) = T(\pi_{\lambda}(g)\cdot 1).
\end{eqnarray}
Let $\phi$ denote a $\Gamma$-invariant eigenfunction of the Laplace
operator with boundary values $T_{\phi}\in\mathcal{D}'(B)$ such that
$\phi=P_{\lambda}(T_{\phi})$. Let $\widetilde{\pi}_{\lambda}$ denote
the dual representation on $\mathcal{D}'(B)$ corresponding to
$\pi_{\lambda}$. Since $\phi$ is invariant, it follows from
\eqref{Poisson intertwines} and the uniqueness of the boundary
values that $T_{\phi}$ is invariant under the actions
$\widetilde{\pi}_{\lambda}(\gamma)$, $\gamma\in\Gamma$.

\subsection*{Regularity of Distribution Boundary Values}

In this subsection we prove a regularity statement for distribution
boundary values corresponding to Laplace eigenfunctions with
eigenvalue parameter $\lambda\in\la^*$ on a compact quotient
$X_{\Gamma}$. These estimates may not be the sharpest possible, but
they are sufficient for our purposes.

Let $T_{\phi}\in \mathcal{D}'(K/M)$ be the (unique) preimage (under
the Poisson transform) of a normalized
$L^2(X_{\Gamma})$-eigenfunction $\phi$ (with exponential growth).
Under the identification $H_{\lambda}^{\infty}\cong C^{\infty}(K/M)$
we view $T_{\phi}$ as a functional on $H_{\lambda}^{\infty}$: For
$f\in H_{\lambda}^{\infty}$ let $T_{\phi}(f)$ be defined by
$T_{\phi}(f_{|K})$. Then $T_{\phi}$ is a continuous linear
functional on $H_{\lambda}^{\infty}$, invariant under
$\widetilde{\pi}_{\lambda}(\gamma)$. As proven in \cite{CG}, Theorem
A.1.4, if $f$ is a smooth vector for the principal series
representation, then $f\in H_{\lambda}^{\infty}$ is a smooth
function on $G$. We consider the mapping
\begin{eqnarray*}
\Phi_{\phi}: H_{\lambda}^{\infty} \rightarrow
C^{\infty}(\Gamma\backslash G), \,\,\,\,\,\,\,\, \Phi_{\phi}(f)(\Gamma g)
= T_{\phi}(\pi_{\lambda}(g)f).
\end{eqnarray*}

\begin{lem}
$\Phi_{\phi}$ is an isometry w.r.t. the norms of $L^2(K/M)$ and $L^2(\Gamma\backslash G)$.
\end{lem}

\begin{proof}
The operator $\Phi_{\phi}$ is equivariant with respect to the
actions $\pi_{\lambda}$ on $H^{\infty}_{\lambda}$ and the right
regular representation of $G$ on $L^2(\Gamma\backslash G)$. We
pull-back the $L^2(\Gamma\backslash G)$ inner product onto the
$(\g,K)$-module $H^{\infty}_{\lambda,K}$ of $K$-finite and smooth
vectors (which is dense in $H^{\infty}_{\lambda}$, \cite{Wal2}, p.
81):
\begin{eqnarray*}
\langle f_1,f_2\rangle_{2} := \langle \Phi_{\phi}(f_1),\Phi_{\phi}(f_2)\rangle_{L^2(\Gamma\backslash G)}.
\end{eqnarray*}
Let $f_1\in H^{\infty}_{\lambda,K}$. Then $A_{f_1}:
H^{\infty}_{\lambda,K} \rightarrow \cc, \, f_2 \mapsto \langle
f_1,f_2\rangle_{2}$ is a conjugate-linear, $K$-finite functional on
the $(\g,K)$-module $H^{\infty}_{\lambda,K}$. This module is
irreducible and admissible, since $H_{\lambda}$ is unitary and
irreducible (\cite{Wal2}, theorems 3.4.10 and 3.4.11). As $A_{f_1}$
is $K$-finite it is nonzero on at most finitely many $K$-isotypic
components. It follows that there is a linear map
$A:H^{\infty}_{\lambda,K}\rightarrow H^{\infty}_{\lambda,K}$ such
that for each $f_1\in H^{\infty}_{\lambda,K}$ the functional
$A_{f_1}$ equals $f_2\mapsto \langle A f_1,f_2\rangle_{L^2(K/M)}$.
The equivariance of $\Phi_{\phi}$ and the unitarity of
$\pi_{\lambda}$ imply that $A$ is $(\g,K)$-equivariant. Using
Schur's lemma for irreducible $(\g,K)$-modules (\cite{Wal2}, p. 80),
we deduce that $A$ is a constant multiple of the identity and hence
$\langle \cdot,\cdot\rangle_{2}$ is a constant multiple of the
original $L^2(K/M)$-inner product on $H_{\lambda,K}^{\infty}$. This
constant is $1$: First, $\Phi_{\phi}(1) = P_{\lambda}(T_{\phi}) =
\phi$ is the $K$-invariant lift of $\phi$ to $L^2(\Gamma\backslash
G)$. Then $\|\Phi_{\phi}(1)\|_{L^2(\Gamma\backslash
G)}=1=\|1\|_{L^2(K/M)}$.
\end{proof}

Let $(y_j)$ and $(x_j)$ be bases for $\lk$ and $\lp$, respectively,
such that $\langle y_j,y_i\rangle = -\delta_{ij}$, $\langle
x_j,x_i\rangle = \delta_{ij}$, where $\langle\,,\,\rangle$ denotes
the Killing form. The Casimir operator of $\lk$ is
$\Omega_{\lk}=\sum_i y_i^2$ and the Casimir operator of $\g$ is
\begin{eqnarray*}
\Omega_{\g} = -\sum_{j}x_j^2 + \Omega_{\lk} \in \mathcal{Z}(\g),
\end{eqnarray*}
where $\mathcal{Z}(\g)$ is the center of the universal enveloping algebra $\mathcal{U}(\g)$ of $\g$.

It follows from $T_{\phi}(f) = \Phi_{\phi}(f)(\Gamma e)$ that
\begin{eqnarray}\label{first estimate}
|T_{\phi}(f)|\leq \|\Phi_{\phi}(f)\|_{\infty}.
\end{eqnarray}
We may now estimate this by a convenient Sobolev norm on
$L^2(\Gamma\backslash G)$. Let $\widetilde{\Delta}$ denote the
Laplace operator of $\Gamma\backslash G$. Then we have
\begin{eqnarray*}
\widetilde{\Delta} = - \Omega_{\mathfrak{g}} + 2\Omega_{\mathfrak{k}},
\end{eqnarray*}
where $\Omega_{\mathfrak{g}}$ and $\Omega_{\mathfrak{k}}$ are the
Casimir operators on $G$ and $K$, respectively.

\begin{defn}
Let $s\in \rr$. The Sobolev space $W^{2,s}(\Gamma\backslash G)$ is
(cf. \cite{Tay81}, p. 22) the space of functions $f$ on
$\Gamma\backslash G$ satisfying $(1+\widetilde{\Delta})^{s/2}(f)\in
L^2(\Gamma\backslash G)$ with norm
\begin{eqnarray*}
\| f \|_{W^{2,s}(\Gamma\backslash G)} = \| (1+\widetilde{\Delta})^{s/2}(f) \|_{L^2(\Gamma\backslash G)}.
\end{eqnarray*}
\end{defn}

Let $m=\dim(\Gamma\backslash G)=\dim(G)$, and let $s>m/2$. The
Sobolev imbedding theorem for the compact space $\Gamma\backslash G$
(\cite{Tay81}, p. 19) states that the identity
$W^{2,s}(\Gamma\backslash G) \hookrightarrow C^0(\Gamma\backslash
G)$ is a continuous inclusion ($C^0(\Gamma\backslash G)$ is equipped
with the usual sup-norm $\|\cdot\|_{\infty}$). It follows that there
exists a $C>0$ such that
\begin{eqnarray}\label{this is 0}
\| \Phi_{\phi}(f)\|_{\infty} \leq C \| \Phi_{\phi}(f)\|_{W^{2,s}(\Gamma\backslash G)} \,\,\,\,\,\,\,\,
\forall f\in C^{\infty}(K/M)  .
\end{eqnarray}

Now we derive the announced regularity estimate for the boundary
values: First, by increasing the Sobolev order, we may assume
$s/2\in\nn$, so
\begin{eqnarray*}
(1+\widetilde{\Delta})^{s/2} =
(1-\Omega_{\mathfrak{g}}+2\Omega_{\mathfrak{k}})^{s/2} \in
\mathcal{U}(\mathfrak{g}).
\end{eqnarray*}
Hence $(1+\widetilde{\Delta})^{s/2}$ commutes with each
$G$-equivariant mapping. Let $f\in H_{\lambda}^{\infty}$. Then
\begin{eqnarray}\label{this is 1}
\left\|\Phi_{\phi}(f)\right\|_{W^{2,s}(\Gamma\backslash G)}
&=& \left\|(1+\widetilde{\Delta})^{s/2}\Phi_{\phi}(f)\right\|_{L^2(\Gamma\backslash G)} \nonumber \\
&=& \left\|\Phi_{\phi}((1-\Omega_{\g}+2\Omega_{\lk})^{s/2}(f))\right\|_{L^2(\Gamma\backslash G)} \nonumber \\
&=& \left\| (1-\Omega_{\g}+2\Omega_{\lk})^{s/2}(f) \right\|_{L^2(K/M)}.
\end{eqnarray}
Recall $\pi_{\lambda}(\Omega_{\mathfrak{k}})=\Delta_{K/M}$ and
$\Omega_{\g}\in \mathcal{Z}(\g)$. Then \eqref{this is 1} equals
\begin{eqnarray}\label{this is 2}
&&\left\| \sum_{k=0}^{s/2} \binom{s/2}{k} (1+2\Delta_{K/M})^{k}
  (-\Omega_{\g})^{s/2-k}(f)\right\|_{L^2(K/M)} \nonumber \\
&& \hspace{3mm} \leq \, \sum_{k=0}^{s/2} \binom{s/2}{k} \left\|
   (1+2\Delta_{K/M})^{k} (-\Omega_{\g})^{s/2-k}(f)\right\|_{L^2(K/M)}.
\end{eqnarray}
Assume $f\in H_{\lambda.K}^{\infty}$ and recall that $\Omega_{\g}$
acts on the irreducible $\mathcal{U}(\g)$-module $H_{\lambda,K}^{\infty}$ by
multiplication with the scalar $-(\langle\lambda,\lambda\rangle +
\langle\rho,\rho\rangle)$ (cf. \cite{Wil}, p. 163), that is
\begin{eqnarray*}
{\Omega_{\g}}_{|H_{\lambda,K}^{\infty}} = -\left(\langle\lambda,\lambda\rangle
+ \langle\rho,\rho\rangle\right) \id_{H_{\lambda,K}^{\infty}}.
\end{eqnarray*}
Then \eqref{this is 2} equals
\begin{eqnarray}\label{this is 3}
\sum_{k=0}^{s/2} \binom{s/2}{k} \left\| (1+2\Delta_{K/M})^k
(|\lambda|^2+|\rho|^2)^{s/2-k}(f) \right\|_{L^2(K/M)}.
\end{eqnarray}
But $\left(|\lambda|^2+|\rho|^2\right)^{-k}\leq 1 + |\rho|^{-s} =:
C'$ ($0\leq k\leq s/2$), so the term in \eqref{this is 3} is bounded
by
\begin{eqnarray}\label{this is 4}
C'\left(|\lambda|^2+|\rho|^2\right)^{s/2} \sum_{k=0}^{s/2}
\binom{s/2}{k} \left\| (1+2\Delta_{K/M})^{k}(f) \right\|_{L^2(K/M)}.
\end{eqnarray}
Since $H_{\lambda.K}^{\infty}$ is dense in $H_{\lambda}^{\infty}$,
this bound holds for all $f\in H_{\lambda}^{\infty}$. Using
\eqref{first estimate}-\eqref{this is 4} we get
\begin{eqnarray}\label{estimate}
|T_{\phi}(f)| \leq C'\left(|\lambda|^2+|\rho|^2\right)^{s/2}
\sum_{k=0}^{s/2} \binom{s/2}{k} \left\| (1+2\Delta_{K/M})^{k}(f)
\right\|_{L^2(K/M)}.
\end{eqnarray}
for all $f\in H_{\lambda}^{\infty}$ and hence for all $f\in
C^{\infty}(K/M)$. We estimate \eqref{estimate} by a continuous
$C^{\infty}(K/M)$-seminorm $\|\cdot\|'$ (independent of $\phi$) and
obtain:

\begin{prop}\label{Proposition Olbrich}
Let $2s>\dim(G)$ such that $s/2\in\nn$. There exists a continuous
$C^{\infty}(B)$-seminorm $\|\cdot\|'$, such that
\begin{eqnarray}
|T_{\phi}(f)|\leq (1+|\lambda|)^s \|f\|' \,\,\,\, \forall \, f\in C^{\infty}(K/M)
\end{eqnarray}
for the distribution boundary values $T_{\phi}$ corresponding to a real-valued and
$L^2(X_{\Gamma})$-normalized eigenfunction $\phi$ of
$\Delta_{\Gamma}$ with eigenvalue $-(|\lambda|^2+|\rho|^2)$.
\end{prop}

Each $f\in C^{\infty}(B)\otimes C^{\infty}(B)$ has the form $f =
\sum_{i,j}c_{i,j}f_i\otimes f_j$. We define a cross-norm
$\|\cdot\|''$ on $C^{\infty}(B)\otimes C^{\infty}(B)$ by
\begin{eqnarray*}
\|f\|'' = \inf\left\{ \sum_{i,j}|c_{i,j}|\|f_i\|'\|f_j\|' : f = \sum_{i,j}c_{i,j}f_i\otimes f_j \right\}.
\end{eqnarray*}
This norm induces a continuous seminorm on the projective tensor
product $C^{\infty}(B)\widehat{\otimes}_{\pi} C^{\infty}(B)$ (cf.
\cite{T}, p. 435). Let $\psi$ denote another normalized
eigenfunction with distribution boundary values
$T_{\psi}\in\mathcal{D}'(B)$ and eigenvalue parameter $\mu\in\la^*$.
Given $f = \sum_{i,j}c_{i,j}f_i\otimes f_j \in C^{\infty}(B)\otimes
C^{\infty}(B)$ we obtain
\begin{eqnarray}\label{take inf}
|(T_{\phi}\otimes T_{\psi})(f)|
&\leq& \sum_{i,j}|c_{i,j}| \cdot |T_{\phi}(f_i)| \cdot |T_{\psi}(f_j)| \nonumber \\
&\leq& (1+|\lambda|)^{s}(1+|\mu|)^{s} \sum_{i,j}|c_{i,j}|\cdot\|f_i\|'\cdot\|f_j\|',
\end{eqnarray}
which implies (by taking the infimum)
\begin{eqnarray}\label{BtimesB}
|(T_{\phi}\otimes T_{\psi})(f)| \leq (1+|\lambda|)^{s}(1+|\mu)^{s} \|f\|''
\end{eqnarray}
for all $f\in C^{\infty}(B)\otimes C^{\infty}(B)$. But
$C^{\infty}(B\times B) \cong C^{\infty}(B) \widehat{\otimes}_{\pi}
C^{\infty}(B)$ (cf. \cite{T}, p. 530) implies that \eqref{BtimesB}
holds for all $f\in C^{\infty}(B\times B)$.

\section{Non-Euclidean Pseudodifferential Operators}\label{section PsiDO}

We use a special $G$-equivariant $\Psi DO$-calculus that generalizes
the \emph{Zelditch quantization} from (\cite{Z84}). In this section
we state some basic definitions and results we need. Full details
will appear in \cite{S}. For the moment, we may drop the rank one
assumption. Fix a co-compact and torsion free discrete subgroup
$\Gamma$ of $G$. Using the identification $X\times B=G/M$ we
identify functions $a(z,\lambda,b)=a(gK,\lambda,g\cdot M)$ on
$X\times\mathfrak{a}^*\times B$ with functions $a(gM,\lambda)$ on
$G/M\times\mathfrak{a}^*$. Let $n=\dim G$ and
$\left\{X_1,...,X_n\right\}$ be a basis for $\mathfrak{g}$ (the
elements are acting on functions on $G/M$ as left-invariant
differential operators). A $\Psi DO$ of order $0$ is a properly
supported operator $A:C_c^{\infty}(X)\rightarrow C_c^{\infty}(X)$
defined by
\begin{eqnarray}
Au(z) = \int_{\mathfrak{a}_+^*}\int_B e^{(i\lambda+\rho)\langle z,b\rangle}
a(z,\lambda,b)\tilde{u}(\lambda,b) \, db \, \dbar\lambda,
\end{eqnarray}
where:
\begin{itemize}
\item[(i)] $\tilde{u}(\lambda,b) = \int_{X}u(x)e^{(-i \lambda+\rho)
\left\langle x,b \right\rangle}dx$ is Helgason's non-euclidean
Fourier transform of $u$ (\cite{He94}, p. 223).
\item[(ii)] $\dbar\lambda=\frac{1}{|W|}|c(\lambda)|^{-2}d\lambda$,
where $|W|$ is the order of the Weyl group.
\end{itemize}
We call $a(z,\lambda,b)$ the \emph{complete symbol} of $A$, which is equivalently given by
\begin{eqnarray}\label{complete symbol}
\left(Ae_{\lambda,b}\right)(z)=a(z,\lambda,b)e_{\lambda,b}(z),
\end{eqnarray}
where for $\lambda\in\la^*$ and $b\in B$ the functions
$e_{\lambda,b}: X\rightarrow\cc, \,\, z\mapsto
e^{(i\lambda+\rho)\langle z,b\rangle}$ are called
\emph{non-Euclidean plane waves}.

Let now $X$ have rank one and denote by $|\cdot|$ the norm on
$\mathfrak{a}^*$ induced by the Killing form. We identify
$\la=\rr=\la^*$: Define $\lambda_0\in\la^*_+$ by
$\lambda_0(X)=\langle X,H_0\rangle$ ($X\in\la$). We always assume
that $a(z,\lambda,b)$ is a classical symbol of order $0$, i.e. it
has an asymptotic expansion of homogeneous symbols of decreasing
order:
\begin{eqnarray}
a(z,\lambda,b) \sim \sum_{j=0}^{\infty} \lambda^{-j} a_{-j}(z,b).
\end{eqnarray}
Asymptotics here means that $a(z,b,\lambda) - \sum_{j=0}^{R}
a_j(z,b) \lambda^{-j+m} \in S^{m-R-1}$, where $a\in
C^{\infty}(X\times\mathfrak{a}^*\times
B)=C^{\infty}(G/M\times\mathfrak{a}^*)$ is a \emph{symbol of order}
$m\in\rr$ ($a\in S^m$) if for all $\beta\in\nn_0$,
$\alpha\in\nn_0^n$ and for each compact subset $C\subset G/M$ it
satisfies
\begin{eqnarray}\label{symbol estimates}\index{symbol estimates}
\|\partial_{\lambda}^{\beta} \, X_1^{\alpha_1}\cdots X_n^{\alpha_n} \,
a(gM,\lambda)\| \leq C_{\beta}(C)(1+|\lambda|)^{m-\beta}.
\end{eqnarray}
We call $\sigma_A := a_0$ the \emph{principal symbol} of $\Op(a)=A$.
Theorems \ref{Intertwining Formula diagonal}, \ref{Intertwining
Formula}, \ref{Asymptotic} only concern principal symbols, so we
often assume that $a$ is independent of $\lambda$.

By $S^m_{\Gamma}$ we denote symbols of order $m$ which are invariant
under the diagonal action of $\Gamma$ on $X\times B$:
\begin{eqnarray}
a(\gamma\cdot z,\lambda,\gamma\cdot b) = a(z,\lambda,b), \,\,\,\,\, \gamma\in\Gamma.
\end{eqnarray}
Let $L^m_{\Gamma}$ be the space of operators associated with such
symbols. If $(T_gu)(z)=u(g\cdot z)$ denotes the translation of
functions on $X$ we find (see \cite{S} for details):
\begin{prop}
Let $a\in S^0$. Then $\mathrm{Op}(a): L^2(X)\rightarrow L^2(X)$ is
continuous. Moreover, $A\in L^m_{\Gamma}$ if and only if $A$
commutes with each $T_{\gamma}$, $\gamma\in\Gamma$.
\end{prop}
Recall from Section \ref{section boundary values} that if $\phi$ is
an eigenfunction of the Laplace operator with eigenvalue
$-(\langle\lambda,\lambda\rangle+\langle\rho,\rho\rangle)$
($\lambda\in\la^*$) and boundary values $T\in\mathcal{D}'(B)$, then
\begin{eqnarray}\label{pull under}
\phi(z) = \int_B e^{(i\lambda+\rho)\langle z,b\rangle} T(db).
\end{eqnarray}
Let $\left\{\phi_{\lambda_j}\right\}$ denote the eigenfunctions of $\Delta_{\Gamma}$ with corresponding
boundary values $T_{\lambda_j}\in\mathcal{D}'(B)$. Then $a\in S^0_{\Gamma}$ induces a bounded operator on $L^2(X_{\Gamma})$ by
\begin{eqnarray}\label{Definition Op(a)}
\Op(a)\phi_{\lambda_j}(z) = \int_B a(z,b) e^{(i\lambda_{j}+\rho)\langle z,b\rangle} T_{\lambda_j}(db),
\end{eqnarray}
where we used the formula
$\Op(a)e^{(i\lambda+\rho)\langle z,b\rangle}=a(z,b)e^{(i\lambda+\rho)\langle z,b\rangle}$
(cf. \eqref{complete symbol}) and pulled the operator under the integral sign in \eqref{pull under}.

\section{Patterson--Sullivan Distributions}\label{PatSul}

In this section we introduce the central concepts we need to
formulate our results: Intermediate values, the Radon transform,
which really is a time average in our context, and the
Patterson--Sullivan distributions.

\subsection*{Intermediate Values}

To motivate the concept of intermediate values, consider the case
where $G/K=PSU(1,1)/PSO(2)$ is the open unit disk $\mathbb{D}$ with
boundary $B=\left\{z\in\cc: |z|=1\right\}$. Let $\gamma\in G$,
$b,b'\in B$. One has the \emph{intermediate value formula} (cf.
\cite{PJN}, p. 8)
\begin{eqnarray}\label{intermediate value formula}
|\gamma(b)-\gamma(b')|^2 = |\gamma'(b)|\cdot|\gamma'(b')|\cdot|b-b'|^2.
\end{eqnarray}
It follows from \cite{He00}, p. 197, that $\frac{d(\gamma\cdot
b)}{db}=e^{-2\rho\langle\gamma\cdot o,\gamma\cdot b\rangle}$, where
$\rho=\frac{1}{2}$. Then
\begin{eqnarray}\label{generalizing}
|\gamma(b)-\gamma(b')|^2
= e^{-\langle\gamma\cdot o,\gamma\cdot b\rangle}
  e^{-\langle\gamma\cdot o,\gamma\cdot b'\rangle} \cdot|b-b'|^2.
\end{eqnarray}

To generalize this we construct certain functions $d_{\lambda}:
G/MA\rightarrow\cc$, which we call \emph{intermediate values}, and
which satisfy a certain equivariance property generalizing
\eqref{generalizing} (cf. \eqref{equivariance dlambda}). This
property then leads to invariance properties of the
Patterson--Sullivan distributions.

\begin{defn}
By \emph{time reversal} we mean the involution
$\iota(x,\xi)=(x,-\xi)$ on the unit cosphere bundle $S^* X$. Under
$\Gamma\backslash G/M= S^* X_{\Gamma}$ the time reversal map takes
the form $\Gamma g\mapsto \Gamma gw$. We say that a distribution $T$
is \emph{time-reversible} if $\iota^*T=T$. Recall that each
$(b,b')\in B^{(2)}$ is of the form $(g\cdot M,g\cdot wM)\in
B^{(2)}$, where $gMA\in G/MA$ is unique. Since $w^2\in M$, time
reversal means
\begin{eqnarray*}
(b,b') = (g\cdot M,g\cdot wM) \mapsto (gw\cdot M,g\cdot w^2M) = (b',b),
\end{eqnarray*}
which is given by $(b,b')\leftrightarrow(b',b)$.
\end{defn}

\begin{defn}
Given $\lambda\in\la^*$, we define $d_{\lambda}: G/MA\rightarrow\cc$ by
\begin{eqnarray}\label{dlambda on G/MA}
d_{\lambda}(gMA) := e^{(i\lambda+\rho)(H(g)+H(gw))}.
\end{eqnarray}
Recall $w^{-1}aw=a^ {-1}$ ($a\in A$), which implies that
$d_{\lambda}$ is well-defined and time reversal invariant. We call
the functions $d_{\lambda}$ \emph{intermediate values}.
\end{defn}

\begin{lem}\label{equivariance property}
Let $\gamma,g\in G$. Then
\begin{eqnarray}\label{equivariance property2}
d_{\lambda}(\gamma g) = e^{(i\lambda+\rho)(\langle\gamma\cdot o,
\gamma g\cdot M\rangle + \langle\gamma\cdot o, \gamma g\cdot
wM\rangle)} d_{\lambda}(g).
\end{eqnarray}
\end{lem}

\begin{proof}
This follows from Lemma \ref{facts}.
\end{proof}

Note that by Lemma \ref{space of geodesics} we may interpret
$d_{\lambda}$ as a function on $B^{(2)}$, that is
\begin{eqnarray*}
d_{\lambda}(b,b') = d_{\lambda}(g\cdot M, g\cdot wM) =
e^{(i\lambda+\rho)(H(g)+H(gw))}
\end{eqnarray*}
for $g=g(b,b')$.

\begin{prop}
$d_{\lambda}(g\cdot M, g\cdot wM) = e^{(i\lambda+\rho)(\langle
g\cdot o, g\cdot M\rangle+\langle g\cdot o,g\cdot wM\rangle)}$.
\end{prop}

\begin{prop}
Let $(b,b')\in B^{(2)}$ and $\gamma\in G$. Then
\begin{eqnarray}\label{equivariance dlambda}
(d_{\lambda}\circ\gamma)(b,b')
= d_{\lambda}(\gamma\cdot b, \gamma\cdot b')
= e^{(i\lambda+\rho)(\langle\gamma\cdot o, \gamma\cdot b\rangle
   +\langle\gamma\cdot o, \gamma\cdot b'\rangle)}d_{\lambda}(b, b').
\end{eqnarray}
\end{prop}
\begin{proof}
Let $g\in G$ such that $(b,b')=(g\cdot M,g\cdot wM)$. Then
$d_{\lambda}(\gamma\cdot b,\gamma\cdot b')=d_{\lambda}(\gamma g)$,
so the assertion follows from Lemma \ref{equivariance property}.
\end{proof}

\subsection*{Invariance Properties}

As in the introduction, let $c_0 \leq c_1\leq c_2\leq \ldots
\rightarrow\infty$ denote the spectrum of $-\Delta_{\Gamma}$ and
$\left\{\phi_{\lambda_j}\right\}$ a fixed
$L^2(X_{\Gamma})$-orthonormal basis of real valued eigenfunctions
with eigenvalues
$c_j=\langle\lambda_j,\lambda_j\rangle+\langle\rho,\rho\rangle\in\rr$.
Then $\lambda_j\in\la^*\cup i\la^*$ and since $c_j\rightarrow\infty$
there are only finitely many $\lambda_j\in i\la^*$, so we may assume
$\lambda_j\in\la^*$ for all $j\in\nn_0$. We only consider
eigenfunctions with exponential growth and denote the corresponding
sequence of distributional boundary values by
$\left\{T_{\lambda_j}\right\}$.

\begin{defn}
The \emph{Patterson--Sullivan distribution $ps_{\lambda_j}$
associated to $\phi_{\lambda_j}$} is the distribution
\begin{eqnarray}\label{ps}
ps_{\lambda_j}(db,db') := d_{\lambda_j}(b,b') \, T_{\lambda_j}(db) T_{\lambda_j}(db').
\end{eqnarray}
on $C_c^{\infty}(B^{(2)})$. The same definition \eqref{ps} extends
$ps_{\lambda_j}$ to a bounded linear functional on the larger space
$d_{\lambda}(b,b')^{-1}\cdot C^{\infty}(B\times B)$.
\end{defn}

\begin{prop}\label{ps invariant}
Suppose that $\phi_{\lambda_j}$ is a $\Gamma$-invariant
eigenfunction of the Laplacian. Let $T_{\lambda_j}$ denote its
boundary values. Then the distribution $ps_{\lambda_j}(db,db')$ is
$\Gamma$-invariant and time reversal invariant.
\end{prop}

\begin{proof}
Time reversibility is obvious. Given a test function $f$ and
$\gamma\in\Gamma$ we have
\begin{eqnarray*}
ps_{\lambda_j}(f\circ\gamma^{-1})
= (T_{\lambda_j}\otimes T_{\lambda_j})(d_{\lambda_j}\cdot(f\circ\gamma^{-1}))
= (\gamma T_{\lambda_j}\otimes \gamma T_{\lambda_j})((d_{\lambda_j}\circ\gamma)\cdot f).
\end{eqnarray*}
It follows from \eqref{boundary values equivariance 1} that
\begin{eqnarray*}
T_{\lambda_j}(d\gamma b)T_{\lambda_j}(d\gamma b')
= e^{-(i\lambda_j+\rho)\langle\gamma\cdot o,\gamma\cdot b\rangle}
  e^{-(i\lambda_j+\rho)\langle\gamma\cdot o,\gamma\cdot b'\rangle}T_{\lambda_j}(db)T_{\lambda_j}(db').
\end{eqnarray*}
Multiplying with \eqref{equivariance dlambda} completes the proof of
$\Gamma$-invariance.
\end{proof}

Recall our notation from \ref{definition g(b,b')}: Let $g(b,b')MA\in
G/MA$ denote the coset corresponding to $(b,b')\in B^{(2)}$.

\begin{defn}
The \emph{Radon transform} on $SX=G/M$ is given by
\begin{eqnarray*}
\mathcal{R}f(b,b') := \int_A f(g(b,b')aM) da,
\end{eqnarray*}
whenever the integral exists. \cite{He00}, p. 91, applied to the subgroup $MA$, yields:
\end{defn}

\begin{lem}\label{Radon compact}
$\mathcal{R}: C_c(SX)\rightarrow C_c(B^{(2)})$.
\end{lem}

\begin{defn}
Let $\mathcal{F}$ denote a bounded fundamental domain for $\Gamma$
in $X$. Following \cite{AZ}, pp. 380-381, we say that $\chi\in
C_c^{\infty}(X)$ is a \emph{smooth fundamental domain cutoff
function} if it satisfies
\begin{eqnarray}\label{smooth fundamental domain cutoff}
\sum_{\gamma\in\Gamma}\chi(\gamma z) = 1 \,\,\,\,\, \forall z\in X.
\end{eqnarray}
Such a function can for example be constructed by taking $\nu\in
C_c^{\infty}(X)$, $\nu=1$ on $\mathcal{F}$, and setting
$\chi(z)=\nu(z)\cdot(\sum_{\gamma\in\Gamma}\nu(\gamma z))^{-1}$. If
$\chi$ satisfies \eqref{smooth fundamental domain cutoff}, then
\begin{eqnarray}
\int_{\mathcal{F}} f \,  dz = \int_X \chi f \, dz, \,\,\,\,\, f\in C(X_{\Gamma}).
\end{eqnarray}
\end{defn}

The following property of these cutoffs is proven in \cite{AZ}, Lemma 3.5:

\begin{prop}\label{independent}
Let $T\in\mathcal{D}'(SX)$ be a $\Gamma$-invariant distribution. Let
$a$ be a $\Gamma$-invariant smooth function on $SX$. Then for any
$a_1,a_2\in\mathcal{D}(SX)$ such that
$\sum_{\gamma\in\Gamma}a_j(\gamma\cdot(z,b))=a(z,b)$ ($j=1,2$) we
have $\langle a_1,T\rangle_{SX}=\langle a_2,T\rangle_{SX}$.
\end{prop}

Given $T$ and $a$ as in Proposition \ref{independent} and if
moreover $\chi_j$ ($j=1,2$) are smooth fundamental domain cutoffs,
then $a_j=\chi_j a$ satisfy the assumptions of the proposition.
Hence $\langle a,T\rangle_{SX_{\Gamma}}:=\langle \chi
a,T\rangle_{SX}$ defines a distribution on the quotient
$SX_{\Gamma}$ and this definition is independent of the choice of
$\chi$.

\begin{defn}\label{Definition Patterson Sullivan}
\begin{itemize}
\item[(1)] The \emph{Patterson--Sullivan distributions} $PS_{\lambda_j}$ on $SX$ are defined by
\begin{eqnarray*}
\langle a,PS_{\lambda_j}\rangle_{SX} := \int_{(B\times
B)\setminus\Delta}(\mathcal{R}a)(b,b') \, ps_{\lambda_j}(db,db').
\end{eqnarray*}
\item[(2)] On $SX_{\Gamma}=\Gamma\backslash SX$ we define the Patterson--Sullivan distributions by
\begin{eqnarray*}
\langle a,PS_{\lambda_j}\rangle_{SX_{\Gamma}} := \langle\chi a,PS_{\lambda_j}\rangle_{SX},
\end{eqnarray*}
where $\chi$ is a smooth fundamental domain cutoff.
\item[(3)] We define normalized Patterson--Sullivan distributions
\begin{eqnarray}\label{normalized}
\widehat{PS}_{\lambda_j} = \frac{1}{\langle 1,PS_{\lambda_j}\rangle_{SX_{\Gamma}}}PS_{\lambda_j},
\end{eqnarray}
which satisfy the normalization condition $\langle
1,\widehat{PS}_{\lambda_j}\rangle_{SX_{\Gamma}}=1$. Note that
$1=\langle 1,W_{\lambda_j}\rangle_{SX_{\Gamma}}$.
\end{itemize}
\end{defn}

In view of Proposition \ref{independent} the definitions made in
\ref{Definition Patterson Sullivan} do not depend on $\chi$.
Consider the expression
\begin{eqnarray*}
PS_{\lambda_j}(a)=\langle a,PS_{\lambda_j}\rangle = \int_{B^{(2)}}
\, d_{\lambda_j}(b,b') \, \mathcal{R}(a)(b,b') \,
T_{\lambda_j}(db)\, T_{\lambda_j}(db').
\end{eqnarray*}
$PS_{\lambda_j}(a)$ is defined if $d_{\lambda_j} \mathcal{R}(a) \in
C^{\infty}(B\times B)$, which is the case for $a\in
C_c^{\infty}(SX)$, since then $\mathcal{R}a\in
C_c^{\infty}(B^{(2)})$, which in turn implies
$d_{\lambda_j}\mathcal{R}(a)\in C_c^{\infty}(B^{(2)}) \subset
C_c^{\infty}(B\times B) = C^{\infty}(B\times B)$.

As an immediate consequence of Proposition \ref{independent} we
obtain:
\begin{prop}
Each $PS_{\lambda_j}$ is a geodesic flow invariant and
$\Gamma$-invariant distribution on $G/M=SX$. On the quotient
$SX_{\Gamma}$, $PS_{\lambda_j}$ still is invariant under the
geodesic flow.
\end{prop}

\subsection*{Proof of Theorem \ref{Intertwining Formula diagonal}}\label{Proof subsection}
\begin{lem}
$L_{\lambda_j}: C_c^{\infty}(G)\rightarrow C_c^{\infty}(G)$.
\end{lem}
\begin{proof}
It is well-known (cf. \cite{He00}, Ch. IV, \S 6, Corollary 6.6) that
\begin{eqnarray}
\rho(H(\overline{n}))\geq0 \,\,\,\,\, \forall \, \overline{n}\in\overline{N}.
\end{eqnarray}
Hence the weight $|e^{-(i\lambda_j+\rho)H(nw)}|\leq C$ is bounded by a constant. The assertion follows from \cite{He00}, p. 91, applied to the closed subgroup $N$ of $G$.
\end{proof}

The following formula is the key tool in the proof of Theorem \ref{Intertwining Formula diagonal}.
\begin{lem}\label{formula}
Let $a\in C^{\infty}(SX)$, $(b,b')\in B^{(2)}$. Then
\begin{eqnarray}\label{special}
\int_X \chi a(z,b) e^{(i\lambda_j+\rho)(\langle z,b\rangle+\langle z,b'\rangle)} \, dz
= d_{\lambda_j}(b,b') \, \mathcal{R}(L_{\lambda_j} \chi a)(b,b').
\end{eqnarray}
\end{lem}
In view of \eqref{general radon special}, \eqref{special} is the
special case $\lambda_j=\lambda_k$ of the more general formula in
Lemma \ref{general version} and hence we do not give a proof here.
Recall that the $\phi_{\lambda_j}$ are real-valued. Let $a\in
C^{\infty}(\Gamma\backslash G/M)$. Then \eqref{Definition Op(a)}
yields
\begin{eqnarray*}
\langle \Op(a)\phi_{\lambda_j}, \phi_{\lambda_j}\rangle =
\int_{B^{(2)}}\left(\int_X \chi a(z,b)e^{(i\lambda_j+\rho)(\langle
z,b\rangle + \langle z,b'\rangle)} \, dz \right) T_{\lambda_j}(db)
\, T_{\lambda_j}(db').
\end{eqnarray*}

It follows from Lemma \ref{formula} that
$d_{\lambda_j}\mathcal{R}(L_{\lambda_j}\chi a)$ has removable
singularities in each $(b,b)\in B\times B$. Hence by the same lemma
$\langle \Op(a)\phi_{\lambda_j}, \phi_{\lambda_j}\rangle$ equals
\begin{eqnarray*}
\langle d_{\lambda_j}\mathcal{R}(L_{\lambda_j}\chi a), T_{\lambda_j}\otimes T_{\lambda_j}\rangle
= \langle \mathcal{R}(L_{\lambda_j}\chi a), ps_{\lambda_j} \rangle
= \langle L_{\lambda_j}(\chi a), PS_{\lambda_j} \rangle,
\end{eqnarray*}
which proves Theorem \ref{Intertwining Formula diagonal}.

\section{Off-diagonal Patterson--Sullivan Distributions}\label{section off-diag}

In this section we generalize the results of Section \ref{PatSul} to
the off-diagonal case and thus prove Theorem \ref{Intertwining
Formula}.

\subsection*{Off-diagonal Intermediate Values}

The construction of $PS_{\lambda_j,\lambda_k}$ is different from the
construction of $PS_{\lambda_j}$. We will see in this section why it
is impossible to define functionals $ps_{\lambda_j,\lambda_k}$
($\lambda_j\neq\lambda_k$).
\begin{defn}
Given $\lambda,\mu\in\la$, define $d_{\lambda,\mu}: G/M\rightarrow\cc$ by
\begin{eqnarray}\label{dlambda mu}
d_{\lambda,\mu}(g) = e^{(i\lambda+\rho)H(g)}e^{(i\mu+\rho)H(gw)}.
\end{eqnarray}
\end{defn}
This is well-defined, since the Iwasawa projection is $M$-invariant
and $M'$ normalizes $M$. What we really need is a geodesic flow
invariant function on $G/M$, that is $d_{\lambda,\mu}$ \emph{should}
be invariant under the right action of $A$. In other words, we would
wish to have $d_{\lambda,\mu}$ well-defined on $G/MA$. But for $g\in
G$, $m\in M$ and $a\in A$ a direct computation shows
\begin{eqnarray}\label{observation}
d_{\lambda,\mu}(gam) = d_{\lambda,\mu}(g) e^{i(\lambda-\mu)\log(a)}.
\end{eqnarray}
It follows that $d_{\lambda,\mu}$ is \emph{not} a function on
$G/MA$. This implies that for $\lambda\neq\mu$ we \emph{cannot}
define a more general function $d_{\lambda,\mu}(b,b')$ in analogy
with \eqref{intermediate value formula}. We will see in
\eqref{repair} how to circumvent this problem. Exactly as in Lemma
\ref{equivariance property} we have:

\begin{lem}\label{equivariance property off-diag}
Let $\gamma,g\in G$. Then
\begin{eqnarray}\label{equivariance property off-diag2}
d_{\lambda,\mu}(\gamma g) = e^{(i\lambda+\rho)\langle\gamma\cdot o,
\gamma g\cdot M\rangle} e^{(i\mu+\rho)\langle\gamma\cdot o,
\gamma g\cdot wM\rangle} d_{\lambda,\mu}(g).
\end{eqnarray}
\end{lem}

\subsection*{Invariance Properties}
Let $f$ be a function on $G/M$ and let $\lambda,\mu\in\la^*$. The
\emph{weighted Radon transform} on $G$ is defined by
\begin{eqnarray}\label{repair}
(\mathcal{R}_{\lambda,\mu}f)(g) := \int_A d_{\lambda,\mu}(ga)f(ga)\,da,
\end{eqnarray}
whenever the integral exists. As in Lemma \ref{Radon compact} we deduce:

\begin{rem}
Let $f\in C_c^{\infty}(G/M)$. Then $\mathcal{R_{\lambda,\mu}}(f)\in
C_c^{\infty}(G/MA)$ is invariant under the geodesic flow of
$G/M=SX$. Hence $\mathcal{R_{\lambda,\mu}}(f)$ is defined on $G/MA$
(see \eqref{observation} and its subsequent remark).
\end{rem}

\begin{defn}
As before, let $g(b,b')\in G$ be a representative for the element in
$G/MA$ that corresponds to $(b,b')\in B^{(2)}$. Given $f\in
C_c^{\infty}(G/M)$, we define
\begin{eqnarray}
\mathcal{R}_{\lambda,\mu}(f)(b,b') := \mathcal{R}_{\lambda,\mu}(f)(g(b,b')).
\end{eqnarray}
Then $\mathcal{R}_{\lambda,\mu}f \in C_c^{\infty}(B^{(2)})$. This
definition is independent of the choice of representative $g(b,b')$,
since $\mathcal{R}_{\lambda,\mu}(f)$ is invariant.
\end{defn}

Let $f\in C_c^{\infty}(G/M)$. The values $|d_{\lambda,\mu}(g)|$ are
independent of $\lambda,\mu$ and all derivatives of
$d_{\lambda,\mu}$ have polynomial growth in $\lambda,\mu$. It
follows that given a continuous seminorm $\|\cdot\|_1$ on
$C^{\infty}(B\times B)$ there exist $K_1>0$ and a continuous
seminorm $\|\cdot\|_2$ on $C_c^{\infty}(G/M)$ such that
\begin{eqnarray}\label{Radon estimate}
\|\mathcal{R}_{\lambda,\mu}(f)\|_1 \leq (1+|\lambda|)^{K_1}(1+|\mu|)^{K_1}\|f\|_2.
\end{eqnarray}

\begin{defn}
The \emph{off-diagonal Patterson--Sullivan distribution associated
to $\phi_{\lambda_j}$ and $\phi_{\lambda_k}$} is the distribution on
$SX=G/M$ defined by
\begin{eqnarray}\label{Off-Diagonal Patterson-Sullivan}
PS_{\lambda_j,\lambda_k}(f):=\langle
f,PS_{\lambda_j,\lambda_k}\rangle := \int_{B^{(2)}} \,
(\mathcal{R}_{\lambda_j,\lambda_k}f)(b,b') \, T_{\lambda_j}(db) \,
T_{\lambda_k}(db').
\end{eqnarray}
\end{defn}

Assume $\mathcal{R}_{\lambda_j,\lambda_k}(f) \in C^{\infty}(B\times
B)$. Then $PS_{\lambda_j,\lambda_k}(f)$ is well-defined. A simple
example is when $f\in C_c^{\infty}(SX)=C_c^{\infty}(G/M)$. In this
case, it follows from \eqref{BtimesB}, \eqref{Radon estimate} and
\eqref{Off-Diagonal Patterson-Sullivan} that there exist $K>0$ and
a continuous seminorm $\|\cdot\|_2$ on $C_c^{\infty}(G/M)$ such that
\begin{eqnarray}\label{define K}
|PS_{\lambda_j,\lambda_k}(f)| \leq (1+|\lambda_j|)^K(1+|\lambda_k|)^K \|f\|_2.
\end{eqnarray}

\begin{rem}\label{general radon special}
Let $(b,b')\in B^{(2)}$ and $g=g(b,b')$. Then
\begin{eqnarray}
\mathcal{R}_{\lambda_j,\lambda_j}(f)(g) = \int_A d_{\lambda_j,\lambda_j}(ga)f(ga)\,da
= d_{\lambda_j}(g(b,b'))(\mathcal{R}f)(b,b'),
\end{eqnarray}
which implies $PS_{\lambda_j,\lambda_j}=PS_{\lambda_j}$.
\end{rem}

\begin{prop}\label{off diagonal invariant}
Suppose that $\phi_{\lambda_j}$ and $\phi_{\lambda_k}$ are
$\Gamma$-invariant eigenfunctions. Then the distribution
$PS_{\lambda_j,\lambda_k}$ on $SX=G/M$ is $\Gamma$-invariant.
\end{prop}

\begin{proof}
Let $f\in C_c^{\infty}(G/M)$ and let $f_{\gamma}$ denote the
translation $f\circ\gamma^{-1}$. Then
\begin{eqnarray*}
\langle f_{\gamma}, PS_{\lambda_j,\lambda_k}\rangle = \int_{B^{(2)}}
\int_A \, d_{\lambda_j,\lambda_k}(g(b,b')a) \, f(\gamma^{-1}g(b,b')a) \, da \,
T_{\lambda_j}(db) \, T_{\lambda_k}(db'),
\end{eqnarray*}
where $(b,b') = (g\cdot M,g\cdot wM)$ for $g=g(b,b')$. By
\eqref{boundary values equivariance 1} this equals
\begin{eqnarray}\label{cancel}
&& \int_{B^{(2)}}\int_A d_{\lambda_j,\lambda_k}(g(\gamma\cdot(b,b'))a) \, f(\gamma^{-1}g(\gamma(b,b'))a)
e^{-(i\lambda_j+\rho)\langle\gamma\cdot o,\gamma\cdot b\rangle} \nonumber\\
&& \hspace{11mm} \times \hspace{2mm} e^{-(i\lambda_k+\rho)\langle\gamma\cdot o,\gamma\cdot b'\rangle} da \,
T_{\lambda_j}(db) \, T_{\lambda_k}(db').
\end{eqnarray}
Recall that $a\in A$ acts trivially on $(M,wM)$. Using this and
\eqref{equivariance property off-diag2} we observe
\begin{eqnarray*}
d_{\lambda_j,\lambda_k}(\gamma ga) = e^{(i\lambda_j+\rho)\langle\gamma\cdot o,\gamma\cdot  b\rangle}
e^{(i\lambda_k+\rho)\langle\gamma\cdot o,\gamma\cdot  b'\rangle} d_{\lambda_j,\lambda_k}(ga).
\end{eqnarray*}
We also have $g(\gamma\cdot(b,b')) = \gamma g(b,b'))$, since
$(b,b')\mapsto g(b,b')\in G/MA$ is $G$-equivariant. Hence
$\gamma^{-1}g(\gamma\cdot(b,b'))=g(b,b')$. Thus we have
\begin{eqnarray*}
\langle f_{\gamma}, PS_{\lambda_j,\lambda_k}\rangle
&=& \int_{B^{(2)}} \int_A d_{\lambda_j,\lambda_k}(g(b,b')a) f(g(b,b')a) \, da \,
    T_{\lambda_j}(db) T_{\lambda_k}(db') \\
&=& \int_{B^{(2)}} \mathcal{R}_{\lambda_j,\lambda_k}f(b,b') \,
    T_{\lambda_j}(db) T_{\lambda_k}(db') = \langle f,
    PS_{\lambda_j,\lambda_k}\rangle,
\end{eqnarray*}
and the proposition follows.
\end{proof}

In view of Proposition \ref{off diagonal invariant}, the definition
of $PS_{\lambda_j,\lambda_k}$ descends to
$SX_{\Gamma}=\Gamma\backslash SX$:
\begin{defn}\label{normalize off-diag}
\begin{itemize}
\item[(1)] The off-diagonal Patterson--Sullivan distributions on
$SX_{\Gamma}$ are defined by ($\chi$ is a smooth fundamental domain cutoff)
\begin{eqnarray}
\langle a,PS_{\lambda_j,\lambda_k}\rangle_{SX_{\Gamma}}
:= \langle\chi a,PS_{\lambda_j,\lambda_k}\rangle_{SX}.
\end{eqnarray}
\item[(2)] We normalize these distributions by
\begin{eqnarray}
\widehat{PS}_{\lambda_j,\lambda_k}
= \frac{1}{\langle 1,PS_{\lambda_k,\lambda_k}\rangle_{SX_{\Gamma}}}PS_{\lambda_j,\lambda_k}.
\end{eqnarray}
\end{itemize}
\end{defn}

\subsection*{Proof of Theorem \ref{Intertwining Formula}}

The following lemma is the off-diagonal analog of Lemma
\ref{formula}.
\begin{lem}\label{general version}
Let $a\in C^{\infty}(SX_{\Gamma})$, $(b,b')\in B^{(2)}$. Then
\begin{eqnarray}\label{intertwining}
\int_X \chi a(z,b)e^{(i\lambda_j+\rho)\langle z,b\rangle}e^{(i\lambda_k+\rho)\langle z,b'\rangle}dz
= \mathcal{R}_{\lambda_j,\lambda_k}(L_{\lambda_k}(\chi a))(b,b').
\end{eqnarray}
\end{lem}

\begin{proof}
Select $g\in G$ such that $(b,b')=(g\cdot M,g\cdot wM)$. The
following manipulations do not depend on the choice of $g$. By
$G$-invariance of $dz$, the left hand side of \eqref{intertwining}
equals
\begin{eqnarray}\label{obtain off}
\int_X \chi a(g\cdot z,b)e^{(i\lambda_j+\rho)\langle g\cdot z,b\rangle}
e^{(i\lambda_k+\rho)\langle g\cdot z,b'\rangle}dz.
\end{eqnarray}
Identify $\chi a$ with a function on $G/M$: Then since $b=g\cdot o$ we have
$$\chi a(gan\cdot o,b)=\chi a(gan\cdot o,g\cdot M)=\chi a(gan\cdot o,gan\cdot M)=\chi a(ganM)$$
(recall that $P=MAN$ fixes $M\in K/M$, in particular $an\in AN$
fixes $M=b_{\infty}$). From the integral formula \eqref{integral
formula AN and G/K} we obtain that \eqref{obtain off} equals
\begin{eqnarray}\label{the integral}
\int_{AN} \chi a(ganM) e^{(i\lambda_j+\rho)\langle gan\cdot o,g\cdot M\rangle}
e^{(i\lambda_k+\rho)\langle gan\cdot o,g\cdot wM\rangle} \, dn \, da.
\end{eqnarray}
But $\langle gan\cdot o,g\cdot M\rangle = \langle gan\cdot
o,gan\cdot M\rangle = H(gan) = H(ga)$ and $H(n^{-1}w)=H(nw)$ (which
is equivalent to $H(\overline{n})=H(\overline{n}^{-1})$ and thus
follows from \cite{He00}, p. 436 (8)). Then \eqref{horocycle
bracket} and \eqref{equivariance} yield
\begin{eqnarray*}
\langle gan\cdot o,g\cdot wM\rangle = -H(n^{-1}a^{-1}w) + H(gw),
\end{eqnarray*}
which by \eqref{w and a} equals $-H(nw) + H(gaw)$. Hence \eqref{the integral} becomes
\begin{eqnarray*}
&& \int_{AN} \chi a(ganM) e^{(i\lambda_j+\rho)H(ga)} e^{(i\lambda_k+\rho)H(gaw)}
   e^{-(i\lambda_k+\rho)H(nw)}dn\, da \\
&=& \int_A d_{\lambda_j,\lambda_k}(gaM) \int_N \chi a(ganM) e^{-(i\lambda_k+\rho)H(nw)} dn\, da \\
&=& \int_A d_{\lambda_j,\lambda_k}(gaM) (L_{\lambda_k}(\chi a))(gaM) da
 = \mathcal{R}_{\lambda_j,\lambda_k}(L_{\lambda_k}(\chi a))(b,b').
\end{eqnarray*}
The independence of the representative $g(b,b')$ follows from the
unimodularity of $A$ and because the mapping $N\ni n\mapsto
\tilde{m}^{-1}n\tilde{m}\in N$ ($\tilde{m}\in M$) preserves the
measure $dn$ (since $M$ is compact).
\end{proof}

As in Section \ref{Proof subsection} we may now integrate
\eqref{intertwining} against $T_{\lambda_j}(db)T_{\lambda_k}(db')$,
which completes the proof of Theorem \ref{Intertwining Formula}.

\section{Proof of Theorem \ref{Asymptotic}}
Given a \emph{phase function} $\psi:\rr^n\rightarrow\cc$ such that
$\mathrm{Im}(\psi)\geq 0$ and an \emph{amplitude} $\alpha\in
C_c^{\infty}(\rr^n)$ and $\tau>0$, consider the integral
\begin{eqnarray*}
I(\tau):= \int e^{i\tau\psi(x)}\alpha(x)dx.
\end{eqnarray*}
It is well known (\cite{OSH}, p. 195) that
if $\psi'\neq0$ on the support of $a$, then
$I(\tau)=O(\tau^{-\infty})$ as $\tau\rightarrow\infty$. Assume that
$0\in\rr^n$ is the only critical point of $\psi$ and let
$H=\psi''(0)$ be nonsingular at $0$. Also assume $\psi(0)=0$. Then
\begin{eqnarray*}
\psi(x) = \langle Hx,x\rangle/2 + O(|x|^3) \hspace{2mm}\textnormal{as } x\rightarrow0
\end{eqnarray*}
and one proves (loc. cit., p. 171) the
asymptotic expansion
\begin{eqnarray}\label{asymptotics}
\int e^{i\tau\psi(x)}\alpha(x)dx \sim C (2\pi/\tau)^{n/2}
\sum_{k=0}^{\infty}\tau^{-k} R_k a(0) \,\,\,\,\,\,\, (\tau\rightarrow\infty),
\end{eqnarray}
where $R_k = (\langle H^{-1}D,D \rangle/{2i})^k$ is a differential
operator on $\rr^n$ of order $2k$ with $D=(D_1,\ldots,D_n)$, where
$D_j=-i\partial_j$ and $C=|\det H|^{-1/2} e^{\pi i\sign(H)/4}$ is a
constant depending on $\psi$. We refer to \eqref{asymptotics} as the
\emph{MSP-formula} (method of stationary phase).

We now assume $\lambda_j\in\la^*_+$ for all $j\in\nn_0$, identify $\la^*_+=\rr^+$,
and write $\la=\rr\cdot H_0$. If $\overline{n} :=
wnw^{-1}\in\overline{N}$ for $n\in N$, then
\begin{eqnarray*}
L_{\lambda}(\chi a)(g)
= \int_{\overline{N}} e^{-(i\lambda)\langle H(\overline{n}), H_0\rangle}
e^{-\rho\langle H(\overline{n}),H_0\rangle}\chi a(gw\overline{n}w^{-1})d\overline{n},
\hspace{4mm} \lambda,\rho\in\rr.
\end{eqnarray*}

\begin{prop}
The phase function $\psi(\overline{n})=\langle H(\overline{n}),
H_0\rangle$ has exactly one critical point, namely $\overline{n}=e$.
The Hessian form at $\overline{n}=e$ is non-degenerate.
\end{prop}
\begin{proof}
\cite{DKV}, pp. 343.
\end{proof}

Clearly $\psi(e)=0$ and for the amplitude $\alpha(\overline{n}) =
e^{-\rho\langle H(\overline{n}),H_0\rangle}\chi
a(gw\overline{n}w^{-1})$ we have $\alpha(e)=\chi a(g)$. Let
$s=\dim(N)=\dim(\overline{N})$. The MSP-formula yields
\begin{eqnarray}\label{integrate this}
L_{\lambda}(\chi a)(g) = C \cdot (2\pi/\lambda)^{s/2} \sum_{n}\lambda^{-n} R_{2n}(\chi a)(g),
\end{eqnarray}
where $R_{2n}$ is a differential operator on $SX$ of order $2n$ and
$R_0$ is the identity. Although we consider off-diagonal elements,
the proof in \cite{AZ} applies with almost no change: Let $K$ be
defined as in \eqref{define K}. Theorem \ref{Intertwining Formula}
implies
\begin{eqnarray*}
&& \langle \Op(a)\phi_{\lambda_j},\phi_{\lambda_k}\rangle_{SX_{\Gamma}}
= \langle L_{\lambda_k}(\chi a),PS_{\lambda_j,\lambda_k}\rangle_{SX} \\
&& \,\,\,\,\, = C \left(2\pi/\lambda_k\right)^{s/2}
\sum_{n=0}^{N} \lambda_k^{-n} \langle R_{2n}(\chi a),PS_{\lambda_j,\lambda_k} \rangle
+ O(\lambda_k^{-N-1+2K}).
\end{eqnarray*}
We choose $N>2K$. Since $R_0$ is the identity, the operator
$L_{\lambda}^{(N)}=\sum_n^{N}\lambda^{-n}R_{2n}$ can be inverted up
to $O(\lambda^{-N-1})$, i.e. one finds differential operators
$M_{\lambda}^{(N)}=\sum_{n=0}^N \lambda^{-n}M_{2n}$, where $M_0=\id$, and $R_{\lambda}^{(N)}$
such that
\begin{eqnarray*}
L_{\lambda}^{(N)} M_{\lambda}^{(N)} = \id + \lambda^{-N-1}R_{\lambda}^{(N)}.
\end{eqnarray*}
An application of Theorem \ref{Intertwining Formula} to $M_{\lambda_k}^{(N)}a$ yields

\begin{eqnarray*}
\langle \Op(M_{\lambda_k}^{(N)}a)\phi_{\lambda_j},\phi_{\lambda_k}\rangle_{SX_{\Gamma}}
&=& \langle L_{\lambda_k}^{(N)}\chi M_{\lambda_k}^{(N)}a,PS_{\lambda_j,\lambda_k}\rangle_{SX}
    + O(\lambda_k^{-N-1+2K}) \\
&=& \langle L_{\lambda_k}^{(N)}M_{\lambda_k}^{(N)}\chi a,PS_{\lambda_j,\lambda_k}\rangle_{SX}
    + O(\lambda_k^{-N-1+2K}) \\
&=& \langle a,PS_{\lambda_j,\lambda_k}\rangle_{SX_{\Gamma}} + O(\lambda_k^{-N-1+2K}).
\end{eqnarray*}

The second line is a consequence of Proposition \ref{independent}. But
\begin{eqnarray}
M_{\lambda}^{(N)}(a) = a + \lambda^{-1}\left( M_2 + \ldots + \lambda^{-N+1}M_{2N}\right)(a),
\end{eqnarray}
so the $L^2$-continuity of zero order pseudodifferential operators implies
\begin{eqnarray}
\langle
\Op(M_{\lambda_k}^{(N)}(a))\phi_{\lambda_j},\phi_{\lambda_k}\rangle_{L^2(X_{\Gamma})}
= \langle
\Op(a)\phi_{\lambda_j},\phi_{\lambda_k}\rangle_{L^2(X_{\Gamma})} +
O(1/\lambda_k),
\end{eqnarray}
which proves
\begin{eqnarray}\label{left side 1}
C \cdot (2\pi/{\lambda_k})^{s/2} \, \langle
a,PS_{\lambda_j,\lambda_k}\rangle_{SX_{\Gamma}} = \langle
\Op(a)\phi_{\lambda_j},\phi_{\lambda_k}\rangle_{SX_{\Gamma}} +
O(1/{\lambda_k}).
\end{eqnarray}
We put $\langle a,PS_{\lambda_j,\lambda_k}\rangle = \langle 1,PS_{\lambda_k,\lambda_k}\rangle \langle a,\widehat{PS}_{\lambda_j,\lambda_k}\rangle$ into \eqref{left side 1} and obtain
\begin{eqnarray}\label{left side 2}
C \cdot (2\pi/\lambda_k)^{s/2} \cdot \langle 1,PS_{\lambda_k,\lambda_k}\rangle \cdot \langle a,\widehat{PS}_{\lambda_j,\lambda_k}\rangle = \langle a,W_{\lambda_j,\lambda_k}\rangle + O(1/{\lambda_k}).
\end{eqnarray}
In particular, for $a=1$, we get
\begin{eqnarray}
C\cdot(2\pi/\lambda_k)^{s/2} \cdot \langle 1,PS_{\lambda_k,\lambda_k}\rangle_{SX_{\Gamma}} = 1 + O(1/{\lambda_k}).
\end{eqnarray}
Together with \eqref{left side 2} this yields
\begin{eqnarray}\label{left side}
\left( 1 + O(1/{\lambda_k}) \right) \cdot \langle a,\widehat{PS}_{\lambda_j,\lambda_k}\rangle = \langle a,W_{\lambda_j,\lambda_k}\rangle + O(1/{\lambda_k}).
\end{eqnarray}
The Wigner distributions and hence by \eqref{left side} the $\langle a,\widehat{PS}_{\lambda_j,\lambda_k}\rangle$ are uniformly bounded. It follows that the left side of \eqref{left side} is asymptotically the same as $\langle a,\widehat{PS}_{\lambda_j,\lambda_k}\rangle$. This completes the proof of Theorem \ref{Asymptotic}.



\addcontentsline{toc}{section}{References}
\providecommand{\bysame}{\leavevmode\hbox to3em{\hrulefill}\thinspace}
\providecommand{\MR}{\relax\ifhmode\unskip\space\fi MR }
\providecommand{\MRhref}[2]{\href{http://www.ams.org/mathscinet-getitem?mr=#1}{#2}}
\providecommand{\href}[2]{#2}

\vspace{30pt}

\textsc{Joachim Hilgert\\Institut f\"ur Mathematik, Universit\"at Paderborn, Warburger Str. 100, 33098 Paderborn, Germany.}\\
\textit{E-mail address:} \texttt{hilgert@math.uni-paderborn.de}\\

\textsc{Michael Schr\"oder\\Institut f\"ur Mathematik, Universit\"at Paderborn, Warburger Str. 100, 33098 Paderborn, Germany.}\\
\textit{E-mail address:} \texttt{michaoe@math.uni-paderborn.de}


\end{document}